\documentclass{article}%
\usepackage{color}
\usepackage{graphicx}
\usepackage{amsmath}
\usepackage{amsfonts}
\usepackage{amssymb}%
\providecommand{\U}[1]{\protect\rule{.1in}{.1in}}
\newtheorem{theorem}{Theorem}

\newtheorem{lemma}[theorem]{Lemma}

\newtheorem{proposition}[theorem]{Proposition}
\newtheorem{remark}[theorem]{Remark}

\newenvironment{proof}[1][Proof]{\textbf{#1.} }{\ \rule{0.5em}{0.5em}}
\begin{document}

\title{\textbf{Differential characters and cohomology of the moduli of flat
Connections}}
\author{\textbf{Marco Castrill\'{o}n L\'{o}pez, Roberto Ferreiro P\'{e}rez}}
\date{}
\maketitle

\begin{abstract}
\noindent Let $\pi\colon P\to M$ be a principal bundle and $p$ an invariant
polynomial of degree r on the Lie algebra of the structure group. The theory of
Chern-Simons differential characters is exploited to define an homology map
$\chi^{k} : H_{2r-k-1}(M)\times H_{k}(\mathcal{F}/\mathcal{G})\to
\mathbb{R}/\mathbb{Z}$, for $k<r-1$, where $\mathcal{F}
/\mathcal{G}$ is the moduli space of flat connections of $\pi$ under the
action of a subgroup $\mathcal{G}$ of the gauge group. The differential characters of first order are related to the Dijkgraaf-Witten action for Chern-Simons Theory. The second order characters are interpreted geometrically as the holonomy of a connection in a line bundle over $\mathcal{F}/\mathcal{G})$. The relationship with
other constructions in the literature is also analyzed.

\end{abstract}

\medskip


\smallskip


\smallskip

\noindent \emph{Acknowledgements:\/} M.C.L. was partially supported by
MINECO (Spain) under grant MTM2015--63612-P.

R.F.P. was partially supported by "Proyecto de investigaci\'{o}n
Santander-UCM PR26/16-20305"

\section{Introduction}

Given a principal bundle and an invariant polynomials $p$ of the Lie algebra
on the structure group, the Chern-Simons differential characters are maps
defined on the space of cycles of the base manifold and taking values in $%
\mathbb{R}/\mathbb{Z}$. Recall that this classical construction is specially
relevant when the characteristic classes of the bundle are trivial, since in
that case the maps vanish on boundaries and descend to homology with values
in $\mathbb{R}/\mathbb{Z}$. They are the so-called secondary Chern-Simons
classes. In this paper we study the Chern-Simons differential characters for
the principal bundle $(P\times \mathcal{A})/\mathcal{G}\rightarrow M\times
\mathcal{A}/\mathcal{G}$, where $P\rightarrow M$ is any principal bundle, $%
\mathcal{A}$ is its affine space of connections and $\mathcal{G}$ is any
subgroup of the group of gauge transformations of $P$ acting freely on $%
\mathcal{A}$. The construction is completed, and this is the main result of
this article, when we restrict the character to the subset $\mathcal{F}%
\subset \mathcal{A}$ of flat connections. Then we prove that, consistently
with the Chern-Simons construction, an homology map
\begin{equation*}
\chi ^{k}:H_{2r-k-1}(M)\times H_{k}(\mathcal{F}/\mathcal{G})\rightarrow
\mathbb{R}/\mathbb{Z},
\end{equation*}%
can be defined any for $k<r-1=\mathrm{deg}(p)$. In the case $k=r-1$ it only
gives a map on cycles $\chi ^{r-1}:H_{r}(M)\times Z_{r-1}(\mathcal{F}/%
\mathcal{G})\rightarrow \mathbb{R}/\mathbb{Z}$ that does not vanish on
boundaries.

The final goal of the definition of $\chi ^{k}$ is the construction of a
tool providing topological information of the moduli space $\mathcal{F}/%
\mathcal{G}$ of flat connection of $P\rightarrow M$. It is well known the
importance of this information in many contexts both in Physics and in
Mathematics: bundles modelling quantum phases as that of the Aharonov-Bohm
effect, topological Field Theories, moduli spaces of Yang-Mills solution on
Riemann surfaces, stability of vector bundles over algebraic varieties, etc.
The study of all these situations are present in many works in the
literature and cherished by many authors. Our contribution provides a step
in that direction.

Along the construction of our main objects, the characteristic classes of $%
(P\times \mathcal{A})/\mathcal{G}$ are regarded through the equivariant
cohomology of the bundle of connections $C(P)\rightarrow M$. One advantage
of working with the equivariant cohomology of $C(P)$ relies on the fact that
it is finite-dimensional and we can easily use local coordinates. We also
obtain the results by working directly with the equivariant cohomology of $%
M\times \mathcal{A}$. The bundle of connections is equipped with universal
characteristic classes for any invariant polynomial $p$. In particular, this
is specially relevant in the case where $2\mathrm{deg}p>\mathrm{dim}M$, that
is, when the corresponding characteristic class vanishes on $M$ because of
dimension considerations. In \cite{BC}, the universal classes are used to
obtain maps $H_{2r-k-1}(M)\times H_{k}(\mathcal{F})\rightarrow \mathbb{R}$,
invariant under the action of the gauge group in $\mathcal{F}$. The
invariance induces an homology map for cycles in $\mathcal{F}/\mathcal{G}$
that are projection of cycles in $\mathcal{F}$. Our work generalizes this
map to arbitrary elements of $H_{\bullet }(\mathcal{F}/\mathcal{G})$, modulo
integers, thus providing a more complete geometric information of the moduli
space.

Now we explain how our results are obtained. In \cite{AS} Chern-Weil theory
is applied to the principal $G$-bundle $(P\times \mathcal{A})/\mathcal{G}%
\rightarrow M\times \mathcal{A}/\mathcal{G}$. A polynomial $p\in I_{\mathbb{Z%
}}^{r}(G)$ determines a cohomology class $c_{p}\in H^{2r}(M\times \mathcal{A}%
/\mathcal{G})$, and by integrating this class on a closed $d$-dimensional
submanifold $c$ of $M$, a cohomology class $\int_{c}c_{p}\in H^{2r-d}(%
\mathcal{A}/\mathcal{G})$ is obtained. Moreover, as $p\in I_{\mathbb{Z}%
}^{r}(G)$, we can also apply the Chern-Simons construction to this bundle.
We use the Cheeger-Simons approach of \cite{Cheeger} based on differential
characters and we denote by $\hat{H}^{k}(M)$ the space of differential
characters of order $k$. If $\mathfrak{A}$ is a connection on the principal $%
\mathcal{G}$-bundle $\mathcal{A}\rightarrow \mathcal{A}/\mathcal{G}$, it
determines a connection $\underline{\mathfrak{A}}$ on $(P\times \mathcal{A})/%
\mathcal{G}\rightarrow M\times \mathcal{A}/\mathcal{G}$ (see below for
details) and hence a differential form $p(\underline{\mathfrak{F}})\in
\Omega ^{2r}(M\times \mathcal{A}/\mathcal{G})$ whose cohomology class is $%
c_{p}$. As $p\in I_{\mathbb{Z}}^{r}(G)$, there exists a differential
character (the Chern-Simons differential character) $\chi _{\underline{%
\mathfrak{A}}}\in \hat{H}^{2r}(M\times \mathcal{A}/\mathcal{G})$ whose
curvature is $p(\underline{\mathfrak{F}})$. As pointed out in \cite{Cheeger}
a Chern-Simons character is determined by a universal characteristic class $%
\Upsilon \in H^{2r}(\textbf{B}G)$ compatible with $p$. The character $\chi _{%
\underline{\mathfrak{A}}}$ defines maps $\chi _{\underline{\mathfrak{A}}%
}^{k}\colon Z_{2r-k-1}(M)\times Z_{k}(\mathcal{A}/\mathcal{G})\rightarrow
\mathbb{R}/\mathbb{Z}$. We study its restriccion to the moduli space of flat
connections and we obtain maps $\chi ^{k}:H_{2r-k-1}(M)\times H_{k}(\mathcal{%
F}/\mathcal{G})\rightarrow \mathbb{R}/\mathbb{Z}$ for $k<r-1$, and $\chi
^{r-1}:H_{r}(M)\times Z_{r-1}(\mathcal{F}/\mathcal{G})\rightarrow \mathbb{R}/%
\mathbb{Z}$ that don't depend on the connection $\mathfrak{A}$ chosen.
Hence, for any $c\in Z_{d}(M)$ we define the differential character $\xi _{c,%
\underline{\mathfrak{A}}}\in \hat{H}^{2r-d}(\mathcal{A}/\mathcal{G})$ by $%
\xi _{c,\underline{\mathfrak{A}}}(s)=\chi _{\underline{\mathfrak{A}}%
}(c\times s)$ and the curvature of $\xi _{c,\underline{\mathfrak{A}}}$ is $%
\int_{c}p(\underline{\mathfrak{F}})$. When we restrict to the moduli space
of flat connections we obtain cohomology classes $\xi _{c}\in H^{2r-d-1}(%
\mathcal{F}/\mathcal{G},\mathbb{R}/\mathbb{Z})$. We obtain similar results
for the moduli space of irreducible flat connections, but using
characteristic classes of the group $\tilde{G}=G/Z(G)$, where $Z(G)$ is the
center of $G$. We show in Section \ref{Sectparticular} that these cohomology
classes can be related to other constuctions in the literature.

Finally, we study the geometric interpretation of the characters $\xi _{c,%
\underline{\mathfrak{A}}}$ of order 1 and 2. The geometrical interpretation
of higher order differential characters is not so simple, and we postpone it
for future research.

As it is well known, a first order differential character can be interpreted
as a function $\mathcal{A}/\mathcal{G}\rightarrow \mathbb{R}/\mathbb{Z}$. In
Section \ref{1orderIntCS} we identify this function with the
Dijkgraaf-Witten action for Chern-Simons theory defined in \cite{DW}. When
restrited to $\mathcal{F}/\mathcal{G}$, we obtain that the Chern-Simons
action is locally constant and that it only depends on the homology class of
the submanifold $c$.

When $d=2r-2$ we have $\xi _{c,\underline{\mathfrak{A}}}\in \hat{H}^{2}(%
\mathcal{A}/\mathcal{G})$ and, by general results on differential
cohomology, it can be represented as the holonomy of a connection on a
Hermitian line bundle over $\mathcal{A}/\mathcal{G}$. By fixing a connection
$A_{0}\in \mathcal{A}$, we define an explicit principal $U(1)$-bundle $%
\mathcal{U}_{c}\rightarrow \mathcal{A}/\mathcal{G}$ with connection $%
\underline{\Theta }_{c}$ and the holonomy of $\underline{\Theta }_{c}$ is
shown to be $\xi _{c,\underline{\mathfrak{A}}}$.  A similar construction is
given in \cite{Freed2} for families of connections, with the assumption that
the bundles are trivial and working with local trivializations. Our result
applies also for non trivial bundles, and we use a global construction.

The character $\xi _{c,\underline{\mathfrak{A}}}$, the connection $%
\underline{\Theta }_{c}$ and the action of $\mathcal{G}\ $on $\mathcal{A}%
\times U(1)$ are defined in terms of a connection $\mathfrak{A}$ on the
principal $\mathcal{G}$-bundle $\mathcal{A}\rightarrow \mathcal{A}/\mathcal{G%
}$. However we prove that the action of $\mathcal{G}\ $on $\mathcal{A}\times
U(1)$ does not depend on the connection $\mathfrak{A}$ chosen.

We study the restriction of this bundle to the moduli space of connections.
Then the connection $\underline{\Theta }_{c}$ does not depend on $\mathfrak{A%
}$. Furthermore,for $r>1$ The holonomy of $\underline{\Theta }_{c}$ only
depends on the homology class of $c$. For $r>2$ the connection is flat and
the holonomy of $\underline{\Theta }_{c}$\ determines the homology class $%
\xi _{c}\in H^{1}(\mathcal{F}/\mathcal{G},\mathbb{R}/\mathbb{Z})$.

Along the article, smooth will mean $C^{\infty }$, $\Omega _{\mathbb{Z}}^{k}$
will denote the space of closed $k$-forms with integral periods, and $%
\mathcal{I}_{\mathbb{Z}}^{r}(G)$ the algebra of invariant polynomials of
degree $r$ on the Lie algebra $\mathfrak{g}$ that are invariant with respect
of the adjoint action of $G$ on $\mathfrak{g}$ and determine an integer
characteristic class.

\section{preliminaries}

\subsection{Chern-Simons differential characters}

Given an entire form $\alpha \in \Omega _{\mathbb{Z}}^{k}(N)$ and a boundary
$c=\partial u$ of a chain $u\in C_{k}(N)$, the integral ${\textstyle\int
\nolimits_{u}}\alpha\operatorname{mod}\mathbb{Z}$ depends only on $c$ and not
on $u$, as if $u^{\prime }$ is another chain such that $c=\partial u^{\prime
}$, then we have
\begin{equation*}
{\int\nolimits_{u^{\prime }}}\alpha ={\int\nolimits_{u}}\alpha +{
\int\nolimits_{u^{\prime }-u}}\alpha ={\int\nolimits_{u}}\alpha\operatorname{mod}
\mathbb{Z}.
\end{equation*}
Hence $\alpha $ defines a homomorphism $\chi \colon B_{k-1}(N)\rightarrow
\mathbb{R}/\mathbb{Z}$. Moreover, if $\alpha $ is the characteristic class
of a bundle, Chern and Simons shown (see \cite{CHS}) that it is always
possible to extend this map to the space of all cycles on $M$ obtaining a
homomorphism $\chi \colon Z_{k-1}(N)\rightarrow \mathbb{R}/\mathbb{Z}$ such
that $\chi (\partial u)=\int\nolimits_{u}\alpha $ for every $u\in C_{k}(N)$.
We now recall this construction (see \cite{CHS} for details). If $A$ is a
connection on a principal $G$-bundle $q\colon Q\rightarrow N$ with curvature
$F$, and $p\in \mathcal{I}_{\mathbb{Z}}^{r}(G)$, we have $p(F)\in \Omega _{%
\mathbb{Z}}^{k}(N)$. The pull-back of this form to $Q$ is exact $q^{\ast
}p(F)=dTp(A)$, where $Tp(A)$ is the Chern-Simons form. Although $Tp(A)$ is
not projectable onto $N$, its reduction modulo $\mathbb{Z}$ satisfyes
\begin{equation*}
\widetilde{Tp(A)}=j^{\ast }\chi _{A}+d\Lambda ,
\end{equation*}%
for certain
\begin{equation*}
\chi _{A}\colon Z_{2r-1}(N)\rightarrow \mathbb{R}/\mathbb{Z}
\end{equation*}%
and $\Lambda \colon C_{2r-2}(Q)\rightarrow \mathbb{R}/\mathbb{Z}$, where $%
\widetilde{z}$ stands for the image of $z\in \mathbb{R}$ under the map $%
\mathbb{R}\rightarrow \mathbb{R}/\mathbb{Z}$. The Chern-Simons construction
can be interpreted in terms of differential characters (see \cite{Cheeger}
and \cite{BB}\ for details). A Cheeger-Simons differential character of
order $k$ is a homomorphism $\chi \colon Z_{k-1}(N)\rightarrow \mathbb{R}/%
\mathbb{Z}$ such that there exist $\alpha \in \Omega ^{k}(N)$ with satisfies
$\chi (\partial u)={\int\nolimits_{u}}\alpha $ for every $u\in C_{k}(N)$. We
say that $\chi $ is a differential character with curvature $\mathrm{curv}%
(\chi )=\alpha $. We denote the space of differential characters of order $k$
on $N$ by $\hat{H}^{k}(N)$. If $f\colon N^{\prime }\rightarrow N$ is a
smooth map, it induces a map $f^{\ast }\colon \hat{H}^{k}(N)\rightarrow \hat{%
H}^{k}(N^{\prime })$ defined by $f^{\ast }\chi (u)=\chi (f\circ u)$.

In general $\chi _{A}$ is not unique, but it can be seen (e.g. see \cite%
{Cheeger}) that the differential character can be determined uniquely by
choosing a universal characteristic class $\Upsilon \in H^{2r}(\textbf{B}G,\mathbb{Z})
$\ on the classifying space $\textbf{B}G$\ compatible with $p$ (i.e. such that $\iota
(\Upsilon )=c_{p}\in H^{2r}(\textbf{B}G,\mathbb{R})$, where $\iota \colon H^{2r}(\textbf{B}G,
\mathbb{Z})\rightarrow H^{2r}(\textbf{B}G,\mathbb{R})$ is the natural map and $c_{p}$
the real universal characteristic class determined by $p\in \mathcal{I}_{%
\mathbb{Z}}^{r}(G)$). We call $\chi _{A}\in \hat{H}^{2r}(N)$ the
Chern-Simons differential character associated to $p$ , $\Upsilon $ and $A$,
and we have $\mathrm{curv}(\chi _{A})=p(F)$. If $A^{\prime }$ is another
connection on $Q$ we have the following identities
\begin{equation*}
p(F_{A^{\prime }})=p(F_{A})+dTp(A^{\prime },A),
\end{equation*}%
and
\begin{equation}
\chi _{A^{\prime }}(c)=\chi _{A}(c)+{\int_{c}}Tp(A^{\prime },A),
\label{trangression}
\end{equation}%
where $Tp(A^{\prime },A)$ is the transgression form
\begin{equation}
Tp(A^{\prime },A)=r\int_{0}^{1}p(A^{\prime }-A,F_{A_{t}},\overset{(r-1}{%
\ldots },F_{A_{t}})dt,  \label{trangression2}
\end{equation}%
with $A_{t}=(1-t)A^{\prime }+tA$. If $Q^{\prime }\rightarrow N^{\prime }$ is
another principal $G$-bundle and $f\colon Q^{\prime }\rightarrow Q$ is a $G$%
-bundle morphism we have, $\chi _{f^{\ast }A}=\underline{f}^{\ast }\chi _{A}$%
, where $\underline{f}\colon N^{\prime }\rightarrow N$ is the induced map on
the bases.

\begin{remark}
The original Chern-Simons and Cheeger-Simons constructions are valid for
finite dimensional manifolds, but they can be extended to Banach or Fr\'{e}%
chet infinite dimensional manifolds (see \cite{BB} for example).
\end{remark}

\section{Equivariant Cohomology in the Cartan model}

Let $\mathcal{G}$ be a connected Lie group acting upon a manifold $N$ and
let
\begin{equation*}
\Omega _{\mathcal{G}}(N)=\left( \mathbf{S}^{\bullet }(\mathrm{Lie\,}\mathcal{%
G}^{\ast })\otimes \Omega ^{\bullet }(N)\right) ^{\mathcal{G}}=\mathcal{P}%
^{\bullet }(\mathrm{Lie\,}\mathcal{G},\Omega ^{\bullet }(N))^{\mathcal{G}}
\end{equation*}%
be the space of $\mathcal{G}$-invariat polynomials on $\mathrm{Lie\,}%
\mathcal{G}$ taking values in $\Omega ^{\bullet }(N)$. We consider the
graduation in $\Omega _{\mathcal{G}}(N)$ given by $\deg (\alpha )=2k+r$ for $%
\alpha \in \mathcal{P}^{k}(\mathrm{Lie\,}\mathcal{G},\Omega ^{r}(N))$, so
that the subspace of $\mathcal{G}$-equivariant differential $q$-forms is
\begin{equation*}
\Omega _{\mathcal{G}}^{q}(N)=\bigoplus_{2k+r=q}(\mathcal{P}^{k}(\mathrm{Lie\,%
}\mathcal{G},\Omega ^{r}(N)))^{\mathcal{G}}.
\end{equation*}%
The Cartan equivariant differential $d_{c}\colon \Omega _{\mathcal{G}%
}^{q}(N)\rightarrow \Omega _{\mathcal{G}}^{q+1}(N)$ is defined as
\begin{equation*}
(d_{c}\alpha )(X)=d(\alpha (X))-i_{X_{N}}\alpha (X),\quad \forall X\in
\mathrm{Lie\,}\mathcal{G},
\end{equation*}%
where $X_{N}\in \mathfrak{X}(N)$ is the vector field defined by the
infinitesimal action of $X$. It is easy to check that $\left( d_{c}\right)
^{2}=0$. The equivariant cohomology in the Cartan model (see for example
\cite{BGV,GuiS}) of $N$ with respect of the action of $\mathcal{G}$ is
defined as the cohomology of this complex.

We now recall the relationship between equivariant cohomology and the
cohomology of the quotient space.\ If the action of $\mathcal{G}$ on $N$ is
free and $N/\mathcal{G}$ is a manifold, then $\pi \colon N\rightarrow N/%
\mathcal{G}$ is a principal $\mathcal{G}$-bundle. Given a principal
connection $A$ in this bundle, we define
\begin{equation*}
\mathrm{C}_{A}(p\otimes \omega )=p(F_{A})\wedge \omega _{\mathrm{hor}A},
\end{equation*}%
for any $p\otimes \omega \in (\mathbf{S}^{\bullet }(\mathrm{Lie\,}\mathcal{G}%
^{\ast })\otimes \Omega ^{\bullet }(N))^{\mathcal{G}}$, where $\beta _{%
\mathrm{hor}A}$ denotes the horizontalization of $\beta \in \Omega ^{\bullet
}(N)$ with respect to the connection $A$. Extending this map by linearity,
we have a sort of Chern-Weil homomorphism
\begin{equation*}
\mathrm{C}_{A}\colon \Omega _{\mathcal{G}}^{\bullet }(N)\rightarrow \left(
\Omega ^{\bullet }(N)\right) _{A\,\mathrm{basic}}\simeq \Omega ^{\bullet }(N/%
\mathcal{G}).
\end{equation*}%
satisfying that $\mathrm{C}_{A}(d_{c}\alpha )=d(\mathrm{C}_{A}(\alpha ))$.
We thus have a map $\mathrm{C}_{A}\colon H_{\mathcal{G}}^{\bullet
}(N)\rightarrow H^{\bullet }(N/\mathcal{G})$, that in fact does not depend
on the connection $A$.

We study the map $\mathrm{C}_{A}$ in the two simplest cases:

\paragraph{Equivariant $1$-forms\label{1equi}}

A $\mathcal{G}$-equivariant $1$-form is simply a $\mathcal{G}$-invariant $1$%
-form $\alpha $, and $\pi ^{\ast }\mathrm{C}_{A}(\alpha )=\alpha _{\mathrm{%
hor}A}=\alpha -\alpha (A)$, where $\alpha (A(Y))=\alpha _{y}(A(Y)_{N})$ for $%
Y\in T_{y}N$. The form $\alpha $ is $D$ closed if $d\alpha =0$ and $\iota
_{X_{N}}\alpha =0$ for every $X\in \mathrm{Lie}\mathcal{G}$. In this case $%
\alpha $ projects onto a closed $1$-form $\underline{\alpha }\in \Omega
^{1}(N/\mathcal{G})$. We conclude that if $\alpha $ is $D$-closed, then $%
\mathrm{C}_{A}(\alpha )=\underline{\alpha }$ does not depend on the
connection $A$ chosen.

\paragraph{Equivariant $2$-forms\label{2equi}}

A $\mathcal{G}$-equivariant $2$-form $\varpi $ is given by $\varpi
(X)=\sigma +\mu (X)$ where $\sigma $ is a $\mathcal{G}$-invariant $2$-form
and $\mu \colon \mathrm{Lie}\mathcal{G}\rightarrow \Omega ^{0}(N)$ a linear $%
\mathcal{G}$-equivariant map. The form $\varpi $ is $D$-closed if $d\sigma =0
$ and $i_{X_{N}}\sigma =\mu (X)$ for every $X$. Hence $\mu $ is a co-moment
map for $\sigma $. We use below the following result

\begin{lemma}
\label{2equiLemma}If $A$ is a connection on $N\rightarrow N/\mathcal{G}$, we
define $\mu (A)\in \Omega ^{1}(N)$ by $\mu (A)(Y)=\mu (A(Y))$ for $Y\in TN$,
and $\mu (A)$ is $\mathcal{G}$-invariant. If $\varpi =\sigma +\mu $ is $D$%
-closed then we have $\pi ^{\ast }\mathrm{C}_{A}(\varpi )=\varpi +D(\mu
(A))=\sigma +d(\mu (A))$.
\end{lemma}

\paragraph{Equivariant characteristic classes }

We also recall here the definition of equivariant characteristic classes of
Berline and Vergne (see \cite{BV1,BV2,BT}). Let $\pi \colon Q\rightarrow N$
a principal $G$-bundle where another Lie group $\mathcal{G}$ acts upon by
automorphisms of this bundle. Let $A$ be a connection on $Q$ invariant under
the action of $\mathcal{G}$. It can be proved (see \cite{BV1}, \cite{BT})
that for every $X\in \mathrm{Lie}\mathcal{G}$ the $\mathfrak{g}$-valued
function $A(X_{P})$ is of adjoint type and defines a section of the adjoint
bundle $v_{A}(X)\in \Omega ^{0}(N,\mathrm{ad}Q)$. For every $p\in \mathcal{I}%
^{r}(G)$ the $\mathcal{G}$-equivariant characteristic form associated to $p$
and $A$, $p_{\mathcal{G},A}\in \Omega _{\mathcal{G}}^{2k}(N)$, is defined as
\begin{align}
p_{\mathcal{G},A}(X)& =p\left( F_{A}-v_{A}(X),\overset{(r}{\ldots }%
,F_{A}-v_{A}(X)\right)   \label{equivariant_form} \\
& =\sum_{i=0}^{r}(-1)^{r-i}\tbinom{r}{i}p(F_{A},\overset{(i}{\ldots }%
,F_{A},v_{A}(X),\overset{(r-i}{\ldots },v_{A}(X))  \notag
\end{align}%
for every $X\in \mathrm{Lie}\mathcal{G}$.

Now we consider the case where $\mathcal{G}$ acts freely on $Q$ and $N$, the
quotients $Q/\mathcal{G}$ and $N/\mathcal{G}$ exist and $Q/\mathcal{G}%
\rightarrow N/\mathcal{G}$ is again a principal $G$-bundle
\begin{equation}  \label{diagrama}
\begin{array}{ccc}
Q & \overset{\bar{\pi}_\mathcal{G}}{\rightarrow} & Q/\mathcal{G} \\
\downarrow &  & \downarrow \\
N & \overset{\pi _\mathcal{G}}{\rightarrow} & N/\mathcal{G}%
\end{array}%
\end{equation}
In this case we have the following result.

\begin{proposition}
\label{quotcha} If $A_{1}$ is a $\mathcal{G}$-invariant connection on $%
Q\rightarrow N$ and $A_{2}$ a connection on the (left) $\mathcal{G}$%
-principal bundle $\pi _{\mathcal{G}}\colon N\rightarrow N/\mathcal{G}$,
then we define the $\mathfrak{g}$-valued $1$-form
\begin{equation*}
A_{1}(A_{2})(\xi )=A_{1}((A_{2}(\pi _{\ast }\xi ))_{Q}),\qquad \xi \in TQ,
\end{equation*}%
Then $A\in \Omega ^{1}(Q,\mathfrak{g})$ defined as $A=A_{1}-A_{1}(A_{2})$ is
projectable to $Q/\mathcal{G}$ and the projection is a connection form $%
\underline{A}$ on $Q/\mathcal{G}\rightarrow N/\mathcal{G}$. The curvature
form of $\underline{A}$ satisfies $\pi _{\mathcal{G}}^{\ast }F_{\underline{A}%
}=(F_{A_{1}})_{\mathrm{hor}A_{2}}-v_{A_{1}}(\pi _{\mathcal{G}}^{\ast
}F_{A_{2}})$. Therefore,
\begin{equation*}
\mathrm{C}_{A_{2}}(p_{\mathcal{G},A_{1}})=p(F_{\underline{A}}).
\end{equation*}
\end{proposition}

\begin{proof}
Given any $X\in\mathrm{Lie}\mathcal{G}$, we have
\begin{equation*}
A(X_{Q})=A_{1}(X_{Q})-A_{1}((A_{2}(X_{N}))_{Q})=A_{1}(X_{Q})-A_{1}(X_{Q})=0.
\end{equation*}
Moreover, $A$ is $\mathcal{G}$-invariant. Indeed, for any $\gamma \in%
\mathcal{G}$ and any $\xi\in\mathfrak{X}(Q)$, we have
\begin{align*}
A(\gamma_{*} \xi) & =A_{1}(\gamma_{*} \xi)-A_{1}((A_{2}(\pi_{*} \gamma
_{*}\xi))_{Q}) \\
& = A_{1} (\gamma_{*}\xi) - A_{1}((A_{2}(\gamma_{*} \pi_{*} \xi))_{Q}) \\
& = A_{1} (\gamma_{*}\xi) - A_{1}((\mathrm{Ad}_{\gamma}A_{2}(\pi_{*}
\xi))_{Q}) \\
& = A_{1} (\gamma_{*}\xi) - A_{1}(\gamma_{*}(A_{2}(\pi_{*} \xi))_{Q}) \\
& = A_{1} (\xi) - A_{1}((A_{2}(\pi_{*} \xi))_{Q}) = A(\xi),
\end{align*}
where we have used the invariance of $A_{1}$ and the fact that $\mathcal{G}$
acts by automorphisms of $Q\to N$. Then, the form $A$ projects to a one form
$\underline{A}\in\Omega^{1} (Q/\mathcal{G},\mathfrak{g})$.

Given $B\in\mathfrak{g}$, we have
\begin{equation*}
A(B_{Q})=A_{1}(B_{Q})-A_{1}((A_{2}(\pi_{\ast}(B_{Q})))_{Q})=A_{1}(B_{Q})=B,
\end{equation*}
so that $\underline{A}(B_{Q/\mathcal{G}})=B$ as the actions of $\mathcal{G}$
and $G$ commute. Furthermore, the form $A$ is $G$-equivariant. To check that
we consider any $g\in G$ and any tangent vector $\xi\in T_{u}Q$. Then
\begin{equation*}
((R_{g})^{\ast}A)(\xi)=A_{1}((R_{g})_{\ast}(\xi))-A_{1}((A_{2}(\pi_{\ast
}(R_{g})_{\ast}\xi))_{Q})
\end{equation*}
where the vector $(A_{2}(\pi_{\ast}(R_{g})_{\ast}\xi))_{Q}=(A_{2}(\pi_{\ast
}\xi))_{Q}$ has to be evaluated in the point $ug$, as $(R_{g})_{\ast}(\xi)%
\in T_{ug}Q$. Since the action of $\mathcal{G}$ and $G$ commute, that vector
is the right translation of the vector $(A_{2}(\pi_{\ast }\xi))_{Q}$
evaluated in $u$. That is
\begin{align*}
((R_{g})^{\ast}A)(\xi) & =A_{1}((R_{g})_{\ast}(\xi))-A_{1}((R_{g})_{\ast
}(A_{2}(\pi_{\ast}\xi))_{Q}) \\
& =\mathrm{Ad}_{g^{-1}}A_{1}(\xi)-\mathrm{Ad}_{g^{-1}}A_{1}((A_{2}(\pi_{\ast
}(\xi)))_{Q})=\mathrm{Ad}_{g^{-1}}A(\xi).
\end{align*}
The commutativity also yields that $(R_{g})^{\ast}\underline{A}=\mathrm{Ad}%
_{g^{-1}}\underline{A}$, that is, $\underline{A}$ is a connection form.

The horizontal lift $\underline{h}:\mathfrak{X}(N/\mathcal{G})\to \mathfrak{X%
}(Q/\mathcal{G})$ defined by the connection $\underline{A}$ is $\underline{h}%
(X)=(\bar{\pi}_{\mathcal{G}})_{*}(h_{1}h_{2}X)$, where $h_{1}$ and $h_{2}$
are the horizontal lifts of $A_{1}$ and $A_{2}$ respectively. Indeed, we
obviously have that the projection $Q/\mathcal{G}\to N/\mathcal{G}$ sends $%
\underline{h}X$ to $X$. On the other hand
\begin{equation*}
\underline{A}(\underline{h}X)=A(h_{1}h_{2}X)=
A_{1}(h_{1}h_{2}X)-A_{1}((A_{2}(h_{2}X))_{Q})=0,
\end{equation*}
for any $X\in\mathfrak{X}(N/\mathcal{G})$. Now, the curvature $F_{%
\underline {A}}$ of $\underline{A}$ understood as a $\mathfrak{g}$-valued
2-form in $Q/\mathcal{G}$ is
\begin{align*}
F_{\underline{A}}(\underline{h}X,\underline{h}Y) & = - \underline {A}([%
\underline{h}X,\underline{h}Y]) = -A([h_{1}h_{2}X,h_{1}h_{2}Y]) \\
& =
-A_{1}([h_{1}h_{2}X,h_{1}h_{2}Y])+A_{1}((A_{2}(%
\pi_{*}[h_{1}h_{2}X,h_{1}h_{2}Y]))_{Q}) \\
& = -A_{1}([h_{1}h_{2}X,h_{1}h_{2}Y])+A_{1}((A_{2}([h_{2}X,h_{2}Y]))_{Q}) \\
& = F_{A_{1}}(h_{1}h_{2}X,h_{1}h_{2}Y)-A_{1}((F_{A_{2}}(h_{2}X,h_{2}Y))_{Q}).
\end{align*}
If we understood the curvature as a adjoint-valued 2-form, we have the
expression in the statement.

Finally, from the definitions of $\mathrm{C}_{A_{2}}$ and $p_{\mathcal{G}%
,A_{1}}$, the last formula of the statement follows.
\end{proof}

\section{The homology map on the moduli space of flat connections\label%
{connections}}

\subsection{Chern-Simons characters on $M\times \mathcal{A}/\mathcal{G}$}

Given a principal $G$-bundle $\pi \colon P\rightarrow M$, let $\mathcal{G}%
\subset \mathrm{Gau}P$ be a subgroup of $\mathrm{Gau}P$ such that $\mathcal{G%
}$ acts freely on the space of connections $\mathcal{A}$ and $\mathcal{A}%
\rightarrow \mathcal{A}/\mathcal{G}$ is a principal $\mathcal{G}$-bundle.
For example we can take $\mathcal{G}$ the subgroup of gauge transformations
fixing a point $p_{0}\in P$ (see e.g. \cite{Donaldson}).

\begin{remark}
To be precise, $\mathcal{A}$ and $\mathrm{Gau}P$ should be defined as the
complections of the spaces of connections and Gauge transformations with
respect to some Sobolev norms (e.g. see \cite{Donaldson}). We do not enter
in details because this norms do not play any role in our constructions.
\end{remark}

We consider the following diagram of principal bundles
\begin{equation*}
\begin{array}{ccc}
P\times \mathcal{A} & \rightarrow  & (P\times \mathcal{A})/\mathcal{G} \\
\downarrow  &  & \downarrow  \\
M\times \mathcal{A} & \rightarrow  & M\times \mathcal{A}/\mathcal{G}%
\end{array}%
\end{equation*}%
where for the vertical projections, the structure group is $G$ acting on the
first factor, and the structure group is $\mathcal{G}$ acting on the second
factor for the horizontal ones. We regard this diagram as the one given in %
\ref{diagrama} with $Q=P\times \mathcal{A}$ and $N=M\times \mathcal{A}$ and
we are going to apply Proposition \ref{quotcha}.

On one hand, if we consider any connection $\mathfrak{A}$ on the principal $%
\mathcal{G}$-bundle $\mathcal{A}\rightarrow \mathcal{A}/\mathcal{G}$, we can
extend it trivially to $M\times \mathcal{A}\rightarrow M\times \mathcal{A}/%
\mathcal{G}$ and we denote it by the same letter. On the other hand, $%
P\times \mathcal{A}\rightarrow M\times \mathcal{A}$ is equipped with a
natural connection 1-form $\mathbb{A}$. Indeed, given any $(X,Y)\in
T_{(p,A)}P\times \mathcal{A}$, we define $\mathbb{A}((X,Y))=A(X)$. The
connection $\mathbb{A}$ is invariant under the acion of the group of
automorphisms of $P$ ,and hence we can define the $\mathcal{G}$-equivariant
characteristic forms $p_{\mathcal{G},\mathbb{A}}$. If $X\in \mathrm{gau}%
P\simeq \Omega ^{0}(M$,$\mathrm{ad}P)$ we have $v_{\mathbb{A}}(X)=\mathrm{pr}%
_{1}^{\ast }X\in \Omega ^{0}(M\times \mathcal{A}$,$\mathrm{ad}(P\times
\mathcal{A})$, where $\mathrm{pr}_{1}\colon M\times \mathcal{A\rightarrow M}$
denotes the projection. Hence we have
\begin{equation}
p_{\mathcal{G},\mathbb{A}}(X)=p(\mathbb{F}-\mathrm{pr}_{1}^{\ast }X,\ldots ,%
\mathbb{F}-\mathrm{pr}_{1}^{\ast }X)  \label{EquiChar}
\end{equation}

According to Proposition \ref{quotcha}, the connections $\mathbb{A}$ and $%
\mathfrak{A}$ determine a connection $\underline{\mathfrak{A}}$ for the
bundle $P\times \mathcal{A}/\mathcal{G}\rightarrow M\times \mathcal{A}/%
\mathcal{G}$ and we denote by $\underline{\mathfrak{F}}$ its curvature.
Given an invariant polynomial $p\in \mathcal{I}^{r}(G)$, we have the
following identity for its associated characteristic form
\begin{equation}
p(\underline{\mathfrak{F}})=\mathrm{C}_{\mathfrak{A}}(p_{\mathcal{G},\mathbb{%
A}})\in \Omega _{\mathbb{Z}}^{2r}(M\times \mathcal{A}/\mathcal{G}).
\label{F2}
\end{equation}

We assume that $p\in I_{\mathbb{Z}}^{r}(G)$, that is, $p$ defines an entire
characteristic class. We also choose that $\Upsilon \in H^{2r}(\textbf{B}G,\mathbb{Z})
$ compatible with $p$ and we apply the Chern-Simons construction to the
connection $\underline{\mathfrak{A}}$. We thus get a differential character
\begin{equation}
\chi _{\underline{\mathfrak{A}}}\colon Z_{2r-1}(M\times \mathcal{A}/\mathcal{%
G})\rightarrow \mathbb{R}/\mathbb{Z}  \label{deff}
\end{equation}%
determined by $\Upsilon $. Furthermore, we denote by
\begin{equation}
\chi _{\underline{\mathfrak{A}}}^{k}\colon Z_{2r-k-1}(M)\times Z_{k}(%
\mathcal{A}/\mathcal{G})\rightarrow \mathbb{R}/\mathbb{Z}  \label{defff}
\end{equation}%
the restriction of the character $\chi _{\underline{\mathfrak{A}}%
}^{k}(c,s)=\chi _{\underline{\mathfrak{A}}}(c\times s)$ to $(2r-k-1,k)$
chains.

By definition, we have
\begin{equation*}
\chi _{\underline{\mathfrak{A}}}(\partial K)=\int_{K}p(F_{\underline{%
\mathfrak{A}}}),
\end{equation*}%
for all $K\in C_{2r}(M\times \mathcal{A}/\mathcal{G})$ and also, for $c\in
Z_{2r-k-1}(M)$, $u\in C_{2r-k}(M)$, $s\in Z_{k}(\mathcal{A}/\mathcal{G})$, $%
t\in C_{k+1}(\mathcal{A}/\mathcal{G})$,
\begin{align}
\chi _{\underline{\mathfrak{A}}}^{k}(\partial u,s)& =\int_{u\times s}p(%
\underline{\mathfrak{F}})=\int_{u\times s}p(\underline{\mathfrak{F}}%
)^{2r-k,k}  \label{delta1} \\
\chi _{\underline{\mathfrak{A}}}^{k}(c,\partial t)& =\int_{c\times t}p(%
\underline{\mathfrak{F}})=\int_{c\times t}p(\underline{\mathfrak{F}}%
)^{2r-k-1,k+1}  \label{delta2}
\end{align}

\subsection{Restriction to the space of flat connections}

We consider the bigraduation of forms $\Omega ^{k}(M\times \mathcal{A}%
)\simeq \bigoplus_{i=0}^{k}\Omega ^{i,j}(M\times \mathcal{A})$ given by the
product structure. The bigraduation is extended to equivariant forms by
setting $\mathrm{pr}_{1}^{\ast }X\in \Omega ^{0,2}$, i.e. if $\alpha \colon
\mathrm{Lie}\mathcal{G}\rightarrow \Omega ^{i,j}(M\times \mathcal{A})$ is a
polynomial of degree $k$, then $\alpha \in \Omega _{\mathcal{G}%
}^{i,j+2k}(M\times \mathcal{A})$. As $\mathfrak{A}$\ comes from a connection
on $\mathcal{A}\rightarrow \mathcal{A}/\mathcal{G}$, its curvature $%
\mathfrak{F}\in \Omega ^{0,2}(M\times \mathcal{A},\mathfrak{g})$ and the map
$\mathrm{C}_{\underline{\mathfrak{A}}}:\Omega _{\mathcal{G}}^{\bullet
}(M\times \mathcal{A})\rightarrow \Omega ^{\bullet }(M\times \mathcal{A}/%
\mathcal{G})$ preserves this bigraduation. Hence we conclude that
\begin{equation}
p(\underline{\mathfrak{F}})^{i,j}=\mathrm{C}_{\mathfrak{A}}(p_{\mathcal{G},%
\mathbb{A}}^{i,j}).  \label{bigrado}
\end{equation}

The curvature of $\mathbb{A}$ can be decomposed $\mathbb{F=F}^{2,0}+\mathbb{F%
}^{1,1}+\mathbb{F}^{0,2}$, where the fact that $\mathbb{F}$ is a form taking
values in $\mathrm{ad}P$ instead of a standard form does not make any
difference. We use the following result (e.g.see \cite[5.2.3]{Donaldson})

\begin{proposition}
\label{FComp}We have $\mathbb{F}^{0,2}=0$, $\mathbb{F}%
_{(x,A)}^{1,1}(a,Y)=a(Y)$, and $\mathbb{F}_{(x,A)}^{2,0}(Y,Y^{\prime
})=F_{A}(Y,Y^{\prime })$ for $Y,Y^{\prime }\in T_{x}M$, and $a,a^{\prime
}\in T_{A}\mathcal{A}\simeq \Omega ^{1}(M,\mathrm{ad}P)$.
\end{proposition}

The subspace of flat connections $\mathcal{F}=\{A\in \mathcal{A}%
:F_{A}=0\}\subset \mathcal{A}$ is invariant under the action of the group $%
\mathcal{G}$. Let $\iota _{\mathcal{F}/\mathcal{G}}:M\times \mathcal{F}/%
\mathcal{G}\rightarrow M\times \mathcal{A}/\mathcal{F}$ denote the
inclusion.

\begin{proposition}
\label{porpnul} We have that
\begin{equation}
\iota _{\mathcal{F}/\mathcal{G}}^{\ast }(p(\underline{\mathfrak{F}}%
)^{2r-j,j})=0,\qquad \forall j<r.  \label{crucial2}
\end{equation}
\end{proposition}

\begin{proof}
By Proposition \ref{FComp}, if $A\in \mathcal{F}$ we have $\mathbb{F}%
_{(x,A)}^{2,0}=0$ for any $x\in M$. Hence on $\mathcal{F}$ we have $\mathbb{F%
}=\mathbb{F}^{1,1}$ and by using equation (\ref{EquiChar}) we obtain%
\begin{align*}
p_{\mathcal{G},\mathbb{A}}|_{_{\mathcal{F}}}(X)& =p(\mathbb{F}^{1,1}-\mathrm{%
pr}_{1}^{\ast }X,\ldots ,\mathbb{F}^{1,1}-\mathrm{pr}_{1}^{\ast }X) \\
& =\sum\limits_{i=0}^{r}(-1)^{i}\tbinom{r}{i}p(\mathrm{pr}_{1}^{\ast }X,%
\overset{(i}{\ldots },\mathrm{pr}_{1}^{\ast }X,\mathbb{F}^{1,1},\overset{(r-i%
}{\ldots },\mathbb{F}^{1,1})\in \bigoplus\limits_{i=0}^{r}\Omega _{\mathcal{G%
}}^{r-i,r+i}(M\times \mathcal{A})
\end{align*}%
as $p(\mathrm{pr}_{1}^{\ast }X,\overset{(i}{\ldots },\mathrm{pr}_{1}^{\ast
}X,\mathbb{F}^{1,1},\overset{(r-i}{\ldots },\mathbb{F}^{1,1})\in \Omega _{%
\mathcal{G}}^{r-i,r+i}(M\times \mathcal{A})$. In particular, $p_{\mathcal{G},%
\mathbb{A}}(X)^{2r-j,j}=0$, for $j<r$. Together with (\ref{bigrado}), we
conclude (\ref{crucial2}).
\end{proof}

We consider the restriction of the character (\ref{defff}) to the moduli
space of flat connections
\begin{equation*}
\left. \chi _{\mathfrak{A}}^{k}\right\vert _{\mathcal{F}/\mathcal{G}}\colon
Z_{2r-k-1}(M)\times Z_{k}(\mathcal{F}/\mathcal{G})\rightarrow \mathbb{R}/%
\mathbb{Z}.
\end{equation*}%
Taking into account (\ref{delta1}), (\ref{delta2}) and Proposition \ref%
{porpnul}, we have
\begin{equation*}
\left. \chi _{\underline{\mathfrak{A}}}^{k}\right\vert _{\mathcal{F}/%
\mathcal{G}}(\partial u,s)=\int_{u\times s}\iota _{\mathcal{F}/\mathcal{G}%
}^{\ast }p(\underline{\mathfrak{F}})=\int_{u\times s}\iota _{\mathcal{F}/%
\mathcal{G}}^{\ast }p(\underline{\mathfrak{F}})^{2r-k,k}=0,
\end{equation*}%
for all $u\in C_{2r-k}(M)$, $s\in Z_{k}(\mathcal{F}/\mathcal{G})$ and $k\leq
r-1$. We also have
\begin{equation*}
\left. \chi _{\underline{\mathfrak{A}}}^{k}\right\vert _{\mathcal{F}/%
\mathcal{G}}(c,\partial t)=\int_{c\times t}\iota _{\mathcal{F}/\mathcal{G}%
}^{\ast }p(\underline{\mathfrak{F}})=\int_{c\times t}\iota _{\mathcal{F}/%
\mathcal{G}}^{\ast }p(\underline{\mathfrak{F}})^{2r-k-1,k+1}=0,
\end{equation*}%
for all $c\in C_{2r-k-1}(M)$, $t\in Z_{k+1}(\mathcal{F}/\mathcal{G})$ and $%
k\leq r-2$. Hence for $k\leq r-2$ the map$\left. \chi _{\underline{\mathfrak{%
A}}}^{k}\right\vert _{\mathcal{F}/\mathcal{G}}$ only depends on the homology
classes and defines a map
\begin{equation*}
\left. \chi _{\underline{\mathfrak{A}}}^{k}\right\vert _{\mathcal{F}/%
\mathcal{G}}\colon H_{2r-k-1}(M)\times H_{k}(\mathcal{F}/\mathcal{G}%
)\rightarrow \mathbb{R}/\mathbb{Z}.
\end{equation*}%
And for $k=r-1$ we have a map
\begin{equation*}
\left. \chi _{\underline{\mathfrak{A}}}^{r-1}\right\vert _{\mathcal{F}/%
\mathcal{G}}\colon H_{r}(M)\times Z_{r-1}(\mathcal{F}/\mathcal{G}%
)\rightarrow \mathbb{R}/\mathbb{Z}.
\end{equation*}

\begin{remark}
For $k=r-1$ we cannot replace $Z_{r-1}(\mathcal{F}/\mathcal{G})$ by $H_{r-1}(%
\mathcal{F}/\mathcal{G})$. For example, if $r=2$ we have $\iota _{\mathcal{F}%
/\mathcal{G}}^{\ast }(\underline{p}_{\mathcal{G}}(X))=p(\mathbb{F}^{1,1}-\pi
^{\ast }X,\mathbb{F}^{1,1}-\pi ^{\ast }X)=p(\mathbb{F}^{1,1},\mathbb{F}%
^{1,1})-2p(\pi ^{\ast }X,\mathbb{F}^{1,1})+p(\pi ^{\ast }X,\pi ^{\ast }X)$,
and the term $p(\mathbb{F}^{1,1},\mathbb{F}^{1,1})\in \Omega _{\mathcal{G}%
}^{2,2}(M\times \mathcal{A})$ is not zero. For example, for $SU(2)$-bundles
over a Riemann surface and $p(X)=\frac{1}{8\pi ^{2}}\mathrm{tr}(X^{2})$, the
form $\sigma =\int_{M}p(\mathbb{F}^{1,1},\mathbb{F}^{1,1})$ gives the
symplectic structure on the moduli space of flat connections defined in \cite%
{AB1} by $\sigma (a,b)=\frac{1}{4\pi ^{2}}\int_{M}\mathrm{tr}(a\wedge b)$
for $a,b\in \Omega ^{1}(M,\mathrm{ad}P)$, and hence $p(\mathbb{F}^{1,1},%
\mathbb{F}^{1,1})$ is not zero.
\end{remark}

\begin{lemma}
\label{Lema1}If $\eta\in\Omega^{0,1}(M\times\mathcal{A}, \mathrm{ad}(P\times%
\mathcal{A}))$  and $F$ is the curvature of the connection $\mathbb{A}+\eta$%
, then we have $F^{2,0}=\mathbb{F}^{2,0}$. In particular, we have ${%
\underline {\mathfrak{F}}}^{2,0}=0$ on $M\times\mathcal{F}/\mathcal{G}$ for
any connection $\mathfrak{A}$ on $\mathcal{A}\rightarrow\mathcal{A}/\mathcal{%
G}$.
\end{lemma}

\begin{proof}
We have $F=\mathbb{F}+d_{\mathbb{A}}\eta+\frac{1}{2}[\eta,\eta]$, and the
result follows using that $[\eta,\eta]\in\Omega^{0,2}(M\times\mathcal{A})$,
and that $(d_{\mathbb{A}}\eta)^{2,0}=(d\eta+[\mathbb{A},\eta])^{2,0}=0$ as $%
\eta\in\Omega^{0,1}(M\times\mathcal{A})$.

The connections $\mathbb{A}$ and $\overline{\pi }_{\mathcal{G}}^{\ast }%
\underline{\mathfrak{A}}$ are connections on the same bundle $P\times
\mathcal{A}\rightarrow M\times \mathcal{A}$. Furthermore, we have $\overline{%
\pi }_{\mathcal{G}}^{\ast }\underline{\mathfrak{A}}=\mathbb{A}+\eta $ with $%
\eta =-\mathbb{A}(\mathfrak{A})$ for $X\in T(M\times \mathcal{A)}$. As $%
\mathfrak{A}$ comes from a connection on $\mathcal{A}\rightarrow \mathcal{A}/%
\mathcal{G}$ we have $\eta \in \Omega ^{0,1}(M\times \mathcal{A},\mathrm{ad}%
(P\times \mathcal{A}))$, and hence $\pi _{\mathcal{G}}^{\ast }{\underline{%
\mathfrak{F}}}^{2,0}=\mathbb{F}^{2,0}=0$ on $M\times \mathcal{F}$. We
conclude that ${\underline{\mathfrak{F}}}^{2,0}=0$.
\end{proof}

The Chern-Simons characters are defined using a connection $\mathfrak{A}$ on
$\mathcal{A}\rightarrow \mathcal{A}/\mathcal{G}$. Fortunately we have the
following result

\begin{theorem}
\label{th} Let $P\rightarrow M$ be a $G$-principal bundle, $p\in \mathcal{I}%
_{\mathbb{Z}}^{r}(G)$ an invariant polynomial of degree $r$, $\Upsilon $ a
compatible characteristic class and $\mathcal{G}\subset \mathrm{Aut}P$ a
subgroup acting freely on the space of connections $\mathcal{A}$. Then,
following the notation above, the map
\begin{equation*}
\left. \chi _{\underline{\mathfrak{A}}}^{k}\right\vert _{\mathcal{F}/%
\mathcal{G}}\colon Z_{2r-k-1}(M)\times Z_{k}(\mathcal{F}/\mathcal{G}%
)\rightarrow \mathbb{R}/\mathbb{Z}.
\end{equation*}%
does not depend on the chosen connection $\mathfrak{A}$ on $\mathcal{A}%
\rightarrow \mathcal{A}/\mathcal{G}$ for $k<r$.
\end{theorem}

\begin{proof}
Let $\mathfrak{A}$, $\mathfrak{A}^{\prime}$ be two connections on $\mathcal{A%
}\rightarrow\mathcal{A}/\mathcal{G}$ and $\underline{\mathfrak{A}}$, $%
\underline{\mathfrak{A}}^{\prime}$ the induced connection on the bundle $%
P\times\mathcal{A}/\mathcal{G}\rightarrow M\times\mathcal{A}/\mathcal{G}$
defined in Proposition \ref{quotcha}.

By the properties of the Chern-Simons differential characters we have
\begin{equation*}
\chi _{\underline{\mathfrak{A}}^{\prime }}(K)=\chi _{\underline{\mathfrak{A}}%
}(K)+\int_{K}Tp(\underline{\mathfrak{A}}^{\prime },\underline{\mathfrak{A}}),
\end{equation*}%
and hence%
\begin{equation}
\chi _{\underline{\mathfrak{A}}^{\prime }}^{k}(c,s)=\chi _{\underline{%
\mathfrak{A}}}^{k}(c,s)+\int_{c\times s}Tp(\underline{\mathfrak{A}}^{\prime
},\underline{\mathfrak{A}})=\chi _{\underline{\mathfrak{A}}%
}^{k}(c,s)+\int_{c\times s}Tp(\underline{\mathfrak{A}}^{\prime },\underline{%
\mathfrak{A}})^{2r-1-k,k}.  \label{CambioC}
\end{equation}

We have $Tp(\underline{\mathfrak{A}}^{\prime },\underline{\mathfrak{A}}%
)=r\int_{0}^{1}p(\underline{\mathfrak{A}}^{\prime }-\underline{\mathfrak{A}}%
,F_{t},\overset{(r-1)}{\ldots },F_{t})dt$ with $\eta =\underline{\mathfrak{A}%
}^{\prime }-\underline{\mathfrak{A}}$ and $F_{t}$ the curvature of $(1-t)%
\underline{\mathfrak{A}}^{\prime }+t\underline{\mathfrak{A}}$. As $\mathfrak{%
A}$ and $\mathfrak{A}^{\prime }$ come from connections on $\mathcal{A}%
\rightarrow \mathcal{A}/\mathcal{G}$ we have $\underline{\mathfrak{A}}%
^{\prime }-\underline{\mathfrak{A}}\in \Omega ^{0,1}(M\times \mathcal{A}/%
\mathcal{G},\mathrm{ad}(P\times \mathcal{A}/\mathcal{G)})$. By Lemma \ref%
{Lema1} we have $F_{t}^{2,0}=0$ on $M\times \mathcal{F}/\mathcal{G}$. Hence $%
Tp(\underline{\mathfrak{A}},\underline{\mathfrak{A}}^{\prime })^{2r-1-k,k}=0$
for $k<r$, and the proof is complete.
\end{proof}

Since the map $\left. \chi _{\underline{\mathfrak{A}}}^{k}\right\vert _{%
\mathcal{F}/\mathcal{G}}$ does not depend on the connection $\mathfrak{A}$,
we simplify the notation and directly write
\begin{equation}
\chi ^{k}\colon H_{2r-k-1}(M)\times H_{k}(\mathcal{F}/\mathcal{G}%
)\rightarrow \mathbb{R}/\mathbb{Z},\qquad k\leq r-2,  \label{mainobject}
\end{equation}%
and
\begin{equation}
\chi ^{r-1}\colon H_{r}(M)\times Z_{r-1}(\mathcal{F}/\mathcal{G})\rightarrow
\mathbb{R}/\mathbb{Z},  \label{mainobject2}
\end{equation}%
as the result of the constructions developed in this section.

Finally, the following result shows that $\chi ^{k}(c,s)$ coincides with the
homology map defined in \cite{BC} when $s\in Z_{\bullet }(\mathcal{F}/%
\mathcal{G})$ is a projection of a cycle $\bar{s}$ in $\mathcal{F}$. We
denote by $\pi _{\mathcal{G}}\colon M\times \mathcal{A}\rightarrow M\times
\mathcal{A}/\mathcal{G}$ the projection.

\begin{proposition}
For $c\in Z_{2r-1-k}(M)$ and $\bar{s}\in Z_{k}(\mathcal{F})$, let $\bar{t}%
\in C_{k+1}(\mathcal{A})$ be a chain that satisfies $\partial \bar{t}=\bar{s}
$ (it exists because $\mathcal{A}$ is contractible) and set $s=\pi _{%
\mathcal{G}}\circ \bar{s}\in Z_{k}(\mathcal{F}/\mathcal{G})$. Then we have
\begin{equation*}
\chi^{k}(c,s)=\int_{c\times \bar{t}}p(\mathbb{F}),\qquad k<r.
\end{equation*}
\end{proposition}

\begin{proof}
We define $t=\pi _{\mathcal{G}}\circ \bar{t}\in C_{k+1}(\mathcal{A}/\mathcal{%
G})$ so that $\partial t=s$. As the connections $\mathbb{A}$ and $\pi _{%
\mathcal{G}}^{\ast }\underline{\mathfrak{A}}$ are connections on the same
bundle $P\times \mathcal{A}\rightarrow M\times \mathcal{A}$, we have $p(%
\mathbb{F})-\pi _{\mathcal{G}}^{\ast }p(\underline{\mathfrak{F}})=d(Tp(%
\mathbb{A},\pi _{\mathcal{G}}^{\ast }\underline{\mathfrak{A}}))$, and using
equation (\ref{delta2}) we obtain
\begin{align*}
\chi _{\underline{\mathfrak{A}}}^{k}(c,s)& =\chi _{\underline{\mathfrak{A}}%
}^{k}(c,\partial t)=\int_{c\times t}p(F_{\underline{\mathfrak{A}}%
})=\int_{c\times \bar{t}}\pi _{\mathcal{G}}^{\ast }p(F_{\underline{\mathfrak{%
A}}}) \\
& =\int_{c\times \bar{t}}p(\mathbb{F})+\int_{c\times \bar{s}}Tp(\mathbb{A}%
,\pi _{\mathcal{G}}^{\ast }\underline{\mathfrak{A}})=\int_{c\times \bar{t}}p(%
\mathbb{F})+\int_{c\times \bar{s}}Tp(\mathbb{A},\pi _{\mathcal{G}}^{\ast }%
\underline{\mathfrak{A}})^{2r-1-k,k}.
\end{align*}%
We have $Tp(\mathbb{A},\pi _{\mathcal{G}}^{\ast }\underline{\mathfrak{A}}%
)=r\int_{0}^{1}p(\eta ,F_{t},\overset{(r-1)}{\ldots },F_{t})dt$ with $F_{t}$
the curvature of $\mathbb{A}+t\eta $. By Proposition \ref{quotcha}\ we have $%
\pi _{\mathcal{G}}^{\ast }\underline{\mathfrak{A}}=\mathbb{A}+\eta $ with $%
\eta =-\mathbb{A}(\mathfrak{A})$ for $X\in T(M\times \mathcal{A)}$. As $%
\mathfrak{A}$ comes from a connection on $\mathcal{A}\rightarrow \mathcal{A}/%
\mathcal{G}$ we have $\eta \in \Omega ^{0,1}(M\times \mathcal{A},\mathrm{ad}%
\mathbf{(}P\times \mathcal{A)})$. By Lemma \ref{Lema1}\ it follows that $%
F_{t}^{2,0}=0$ on $M\times \mathcal{F}$, and hence $Tp(\mathbb{A},\pi _{%
\mathcal{G}}^{\ast }\underline{\mathfrak{A}})^{2r-1-k,k}=0$ for $k<r$.
\end{proof}

\section{Construction using the bundle of connections}

In the preceding section our basic result Proposition \ref{porpnul} is
proved using the expression of the curvature $\mathbb{F}$ in the infinite
dimensional manifold $M\times \mathcal{A}$. In this Section we show that in
accordance with the results of \cite{equiconn} Proposition \ref{porpnul} can
be also obtained by studying the equivariant characteristic classes of the
finite dimensional bundle of connections $C(P)$. In our case both approaches
are equivalent, but we recall that for example, in the study of locality in
anomaly cancellation it is fundamental the fact the equivariant forms can be
obtained from the bundle of connections (see \cite{anomalies} for details).

\subsection{Bundle of connections\label{sec1}}

Given a fiber bundle $\pi: E\to M$, an Ehresmann connection is a section of
the jet bundle $J^{1} E\to E$ of local sections of $E\to M$. Geometrically,
this section $\varsigma:E\to J^{1}E$ assigns to every point $y\in E$, a
complementary subspace $H_{y}E=s_{*}(T_{x} M)\subset T_{y}E$, of the
vertical subspace $V_{y}E\subset T_{y} E$, where $\varsigma(y) = j^{1} _{x} s
$, $s(x)=y$ and $V_{y}E=\mathrm{ker}\,\pi_{*}$. If the bundle $E\to M$ is a $%
G$-principal bundle $P\to M$, connections are assumed to be $G$-invariant in
the sense that $(R_{g})_{*}H_{y}P=H_{R_{g}(y)}P$, for every $g\in G$, where $%
R_{g}:P\to P$ is the right action of $G$ on $P$. Therefore, principal
connections are sections of the quotient bundle $(J^{1}P)/G\to P/G=M$, where
the action $R_{g}$ lifts to the jet space in the natural way, i.e., $%
R_{g}(j^{1}_{x} s)=j^{1}_{x} (R_{g}\circ s)$. This bundle is called the
bundle of connections and is denoted by $q: C(P)\to M$.

Since, on one hand, $J^{1}P\to P$ is an affine bundle modeled by the vector
bundle $T^{*}M\otimes VP\to P$ and, on the other hand, we have a natural
identification
\begin{equation}  \label{idd}
P\times\mathfrak{g}= VP, \qquad(y,\xi)\mapsto\left. \frac {d}{d\epsilon}%
\right| _{\epsilon=0} R_{\mathrm{exp}\epsilon\xi}y,
\end{equation}
the bundle of connections is also an affine bundle, modeled by the vector
bundle $T^{*}M\otimes\mathrm{ad}P\to M$, where $\mathrm{ad}P=(P\times
\mathfrak{g})/G$ is the adjoint bundle, that is, the associated bundle to $P$
by the adjoint action of $G$ on its Lie algebra $\mathfrak{g}$. Then we have
that $\mathrm{dim}C(P)=n + nm$, $n=\mathrm{dim}M$, $m=\mathrm{dim}G$.

The principal $G$-bundle $J^{1} P\to C(P)$ is equipped with a canonical
(tautological) principal connection form $\mathbf{A}$ (for example, see \cite%
{geoconn}). This form is the contact form of the jet bundle defined as $%
\mathbf{A}(v)=\xi\in\mathfrak{g}$, $v\in T_{j^{1}_{x} s} (J^{1}P)$, where $%
\xi$ is the element of the Lie algebra associated to $(\pi_{10})_{*}v-(s%
\circ\pi_{1})_{*}v\in V_{s(x)}P$ by the isomorphism \eqref{idd}, and the
source and target projections of the jet bundle are $\pi_{1} :J^{1}P\to M$, $%
\pi_{10}:J^{1}P\to P$ respectively. The connection form $\mathbf{F}$ of $%
\mathbf{A}$ can be seen as a 2-form in $C(P)$ taking values in the adjoint
bundle of $J^{1}P\to C(P)$. It is easy to check that the bundles $q^{*}P\to
C(P)$ and $J^{1}P\to C(P)$ are naturally diffeomorphic, so that $\mathbf{F}%
\in\Omega(C(P),q^{*}\mathrm{ad}P)$. This canonical curvature form satisfies
that, given a section $\sigma:M\to C(P)$ of the bundle of connections, that
is, a connection $A_{\sigma}$ on $P\to M$, the pull-back $\sigma ^{*}(%
\mathbf{F})\in\Omega^{2}(M,\mathrm{ad}P)$ is the curvature 2-form $%
F_{A_{\sigma}}$ of $A_{\sigma}$.

Coordinates in $C(P)$ are defined as follows. Let $(x^{1},...,x^{n})$ be a
coordinate system on a domain $U\subset M$ where the principal bundle is
trivializable, $\pi^{-1}(U)=U\times G$. The trivialization induces a (flat)
connection on $\pi^{-1}(U)\to U$ and therefore a section $\sigma_{U} : U\to
C(P)|_{U}$ of the bundle of connection. With this section, we identify the
affine bundle $C(P)|_{U}\to U$ with $(T^{*}M\otimes\mathrm{ad}P)|_{U} \to U$%
. If $(B_{1},\ldots,B_{m})$ is a basis of $\mathfrak{g}$, the set $(\tilde {B%
}_{1},...,\tilde{B}_{m})$ is a basis of section of $\mathrm{ad}P|_{U}\to U$,
where $\tilde{B}_{i}(x)=((x,g),\mathrm{Ad}_{g}B_{i})_{G}\in\mathrm{ad}P_{x}$%
, $x\in U$. Coordinates $(A^{\alpha}_{j})$, $j=1,...,n$, $\alpha= 1,...,m$,
are defined as $A = A^{\alpha}_{j} (A) dx^{j}\otimes\tilde{B}_{\alpha}$ for
any $A \in T^{*}M\otimes\mathrm{ad}P|_{U}\simeq C(P)|_{U}$. It is not very
complicated to check (for instance, see \cite{geoconn}) that in this
coordinate system the canonical curvature from takes the expression
\begin{equation*}
\mathbf{F}=\left( dA_{j}^{\alpha}\wedge
dx^{j}+c_{\beta\gamma}^{\alpha}(A_{i}^{\beta}A_{j}^{\gamma}-A_{j}^{%
\beta}A_{i}^{\gamma})dx^{i}\wedge dx^{j}\right) \otimes\tilde{B}_{\alpha}.
\end{equation*}

Given any fiber bundle $\pi :E\rightarrow M$, a differential form $\eta \in
\Omega ^{\bullet }(J^{1}E)$ is said to be horizontal or $0$-contact if $%
i_{Y}\eta =0$ for every vector $Y\in T(J^{1}E)$ vertical with respect to the
projection $\pi _{1}:J^{1}E\rightarrow M$. On the other hand, $\eta $ is
said to be contact if $(j^{1}s)^{\ast }\eta =0$, for every (local) section $s
$ of $\pi $. Given any $r$-form $\omega $ on $E$, its pull-back $\pi
_{10}^{\ast }\omega \in \Omega ^{\bullet }(J^{1}E)$ can be decomposed as the
sum of a horizontal and a contact form in a unique way. This splitting can
be further refined. A contact form is said to be $k$-contact if for every
vertical vector $Y$, the form $i_{Y}\omega $ is $(k-1)$-contact. Then, as
proved in \cite{krupka}, every $r$-form $\omega \in \Omega ^{\bullet }(E)$
admits a unique decomposition
\begin{equation*}
\pi _{10}^{\ast }\omega =\sum_{s=0}^{r}\omega ^{s,r-s},
\end{equation*}%
where $\omega ^{s,r-s}\in \Omega ^{\bullet }(J^{1}E)$ is $s$-contact. In the
case where $E=C(P)$, it is easy to check that the decomposition of the
canonical curvature $\mathbf{F}$ reads
\begin{align}
\mathbf{F}^{2,0}& ={\textstyle\sum\limits_{i<j}}\left( A_{j,i}^{\alpha
}-A_{i,j}^{\alpha }+c_{\beta \gamma }^{\alpha }(A_{i}^{\beta }A_{j}^{\gamma
}-A_{j}^{\beta }A_{i}^{\gamma })\right) dx^{i}\wedge dx^{j}\otimes \tilde{B}%
_{\alpha },  \label{flatcoord} \\
\mathbf{F}^{1,1}& =(dA_{i}^{\alpha }-A_{i,j}^{\alpha }dx^{j})\wedge
dx^{i}\otimes \tilde{B}_{\alpha },  \notag \\
\mathbf{F}^{0,2}& =0,  \notag
\end{align}%
where the fact that $\mathbf{F}$ is a form taking values in $q^{\ast }%
\mathrm{ad}P$ instead of a standard form does not make any difference.

We consider the evaluation map
\begin{eqnarray*}
\mathrm{ev}\colon M\times \mathcal{A} &\rightarrow &C(P) \\
(x,A) &\mapsto &\sigma _{A}(x).
\end{eqnarray*}

\begin{proposition}
Let $\mathbf{A}$ be the canonical connection on the bundle $%
J^{1}P\rightarrow C(P)$. For any subgroup $\mathcal{G}\subset \mathrm{Aut}P$%
, we consider its natural action on $J^{1}P\rightarrow C(P)$. Then, we have
the following identity of equivariant characteristic forms%
\begin{equation*}
\mathrm{ev}^{\ast }p_{\mathcal{G},\mathbf{A}}=p_{\mathcal{G},\mathbb{A}}.
\end{equation*}
\end{proposition}

\begin{proof}
The map $\mathrm{ev}$ is $\mathrm{Aut}P$-equivariant so that it defines a
map
\begin{equation*}
\mathrm{ev}^* :\Omega ^\bullet _\mathcal{G}(C(P)) \to \Omega ^\bullet _%
\mathcal{G}(M\times \mathcal{A}).
\end{equation*}
If we consider another evaluation map $\overline{\mathrm{ev}}:P\times
\mathcal{A}\to J^1P$, $\overline{\mathrm{ev}}(p,A)=\mathrm{Hor}^A_p$, where $%
\mathrm{Hor}^A_p :T_xM\to T_pP$, $\pi(p)=x$, is the horizontal lift to $p$
defined by $A$, we have the diagram
\begin{equation*}
\begin{array}{ccc}
P\times \mathcal{A} & \overset{\overline{\mathrm{ev}}}{\rightarrow} & J^1P
\\
\downarrow &  & \downarrow \\
M\times \mathcal{A} & \overset{\mathrm{ev}}{\rightarrow} & C(P)%
\end{array}%
\end{equation*}
It is easy to check that the pullback of the connection form $\overline{%
\mathrm{ev}}^*\mathbf{A}$ is exactly the connection form $\mathbb{A}$ in $%
P\times \mathcal{A}\to M\times\mathcal{A}$ introduced above. This completes
the proof.
\end{proof}

The evaluation map can be naturally extended to the jet space $J^{1}C(P)$
\begin{align*}
\mathrm{ev}_{1}& \colon M\times \mathcal{A}\rightarrow J^{1}C(P) \\
(x,A)& \mapsto j_{x}^{1}A
\end{align*}%
and (e.g. see \cite{VB}) the map $\mathrm{ev}_{1}$ is compatible with the
bigraduation of forms on $M\times \mathcal{A}$ given by the product
structure and the bigraduation $\Omega ^{i,j}(C(P))$ on $J^{1}C(P)$ of
contact forms, that is,
\begin{equation*}
\mathrm{ev}_{1}^{\ast }(\Omega ^{i,j}(C(P))\subset \Omega ^{i,j}(M\times
\mathcal{A}).
\end{equation*}

The bigraduation is extended to equivariant forms by setting $\pi ^{\ast
}X\in \Omega ^{0,2}$, i.e. if $\alpha \colon \mathrm{Lie}\mathcal{G}%
\rightarrow \Omega ^{i,j}(C(P)\mathcal{)}$ is a polynomial of degree $k$,
then $\alpha \in \Omega _{\mathcal{G}}^{i,j+2k}(C(P))$, and the same
definition for forms on $M\times \mathcal{A}$. As $\mathfrak{A}$\ comes from
a connection on $\mathcal{A}\rightarrow \mathcal{A}/\mathcal{G}$ the map $%
\mathrm{C}_{\underline{\mathfrak{A}}}:\Omega _{\mathcal{G}}^{\bullet
}(M\times \mathcal{A})\rightarrow \Omega ^{\bullet }(M\times \mathcal{A}/%
\mathcal{G})$ also preserves this bigraduation. Hence we conclude that $p(%
\underline{\mathfrak{F}})^{i,j}=\mathrm{C}_{\mathfrak{A}}(\mathrm{ev}%
_{1}^{\ast }p_{\mathcal{G},\mathbf{A}}^{i,j})$.

We consider the subset $\mathcal{E}_{\mathcal{F}}=\{j_{x}^{1}A\in J^{1}C(P):%
\mathbf{F}_{j_{x}^{1}A}^{2,0}=0\}\subset J^{1}C(P)$ of $1$-jets of flat
connections. Obviously, a connection $A\in \Gamma (M,C(P))$ is flat if and
only if $j_{x}^{1}A\in \mathcal{E}_{\mathcal{F}}$ for every $x\in M$. On $%
\mathcal{E}_{\mathcal{F}}$ we have
\begin{align*}
p_{\mathcal{G},\mathbf{A}}|_{\mathcal{E}_{\mathcal{F}}}(X)& =p(\mathbf{F}%
^{1,1}-\pi ^{\ast }X,\ldots ,\mathbf{F}^{1,1}-\pi ^{\ast }X) \\
& =\sum\limits_{i=0}^{r}(-1)^{i}\tbinom{r}{i}p(\pi ^{\ast }X,\overset{(i}{%
\ldots },\pi ^{\ast }X,\mathbf{F}^{1,1},\overset{(r-i}{\ldots },\mathbf{F}%
^{1,1})\in \bigoplus\limits_{i=0}^{r}\Omega ^{r-i,r+i}(C(P))
\end{align*}%
as $p(\pi ^{\ast }X,\overset{(i}{\ldots },\pi ^{\ast }X,\mathbf{F}^{1,1},%
\overset{(r-i}{\ldots },\mathbf{F}^{1,1})\in \Omega ^{r-i,r+i}(C(P))$. In
particular, $p_{\mathcal{G},\mathbf{A}}(X)^{2r-j,j}=0$, for $j<r$, and
Proposition \ref{porpnul} follows.

\section{Cohomology classes on the moduli space of irreducible flat
connections}

If we denote by $\widetilde{\mathcal{A}}$ the space of irreducible
connections, then $\widetilde{\mathcal{A}}/\mathrm{Gau}P$ is a well defined
manifold and we can try to apply the preceding construction to this case.
The unique obstacle to do it is that the action of $\mathrm{Gau}P$ on $%
\widetilde{\mathcal{A}}$ is not free. If we denote by $Z(G)$ the center of $G
$, then the isotropy of any $A\in \widetilde{\mathcal{A}}$ is $Z(G)$
(considered as constant gauge transformations). Hence, if we define the
group $\widetilde{\mathcal{G}}=\mathrm{Gau}P/Z(G)$ then $\widetilde{\mathcal{%
G}}$ acts freely on $\widetilde{\mathcal{A}}$ and $\widetilde{\mathcal{A}}%
\rightarrow \widetilde{\mathcal{A}}/\widetilde{\mathcal{G}}$ is a principal $%
\mathcal{\tilde{G}}$-bundle (see \cite{Donaldson}). Note that $\widetilde{%
\mathcal{A}}/\widetilde{\mathcal{G}}=\widetilde{\mathcal{A}}/\mathrm{Gau}P$
an hence this quotient is the moduli space of irreducible connections. The
problem now is that $\widetilde{\mathcal{G}}$ does not act on $P\times
\widetilde{\mathcal{A}}$ because $Z(G)$ does not act trivially on $P$. To
solve this problem we define $\widetilde{\mathbb{P}}=(P\times \widetilde{%
\mathcal{A}})/Z(G)=\mathbb{(}P/Z(G))\times \widetilde{\mathcal{A}}$ which is
a principal $\widetilde{G}$-bundle, where $\tilde{G}=G/Z(G)$.

Let $\mathfrak{z}$ be the lie algebra of $Z(G)$ and $\widetilde{\mathfrak{g}}%
=\mathfrak{g}/\mathfrak{z}$ the Lie algebra of $\widetilde{G}$. We denote by
$t\colon\mathfrak{g}\rightarrow\widetilde{\mathfrak{g}}$ the projection. The
connection $\mathbb{A}\in\Omega^{1}(P\times\mathcal{A},\mathfrak{g})$
induces a connection $\widetilde{\mathbb{A}}$ on $\widetilde{\mathbb{P}}$
which is invariant under the action of $\widetilde{\mathcal{G}}$ in the
following way.

\begin{lemma}
The form $t\circ\mathbb{A}\in\Omega^{1}(P\times\mathcal{A},\widetilde{%
\mathfrak{g}})$ projects to a connection form $\widetilde{\mathbb{A}}%
\in\Omega^{1}(\widetilde{\mathbb{P}},\widetilde{\mathfrak{g}})$ which is
invariant under the action of $\widetilde{\mathcal{G}}$.
\end{lemma}

\begin{proof}
On one hand, given $B\in\mathfrak{z}$, we have $t\circ\mathbb{A}(B_{\mathbb{P%
}})=t(B)=0$. On the other, for $g\in Z(G)$, $R_{g}^{\ast}(t\circ\mathbb{A}%
)=t\circ R_{g}^{\ast}\mathbb{A}=t\circ\mathrm{Ad}_{g^{-1}}\mathbb{A}=t\circ%
\mathbb{A}$. Hence $t\circ\mathbb{A}=\pi_{Z}^{\ast}\tilde{\mathbb{A}}$,
where $\pi_{Z}:\mathbb{P}\rightarrow\tilde{\mathbb{P}}$. As the fundamental
vector fields $B_{\mathbb{P}}$, $B\in\mathfrak{g}$, project to $(t(B))_{%
\tilde{\mathbb{P}}}$, we have that $\tilde{\mathbb{A}}(D_{\tilde{\mathbb{P}}%
})=D$, for any $D\in\widetilde{\mathfrak{g}}$. Moreover, given $(g)_{Z}\in%
\tilde{G}$, since $t\circ\mathrm{Ad}_{g}=\mathrm{Ad}_{(g)_{Z}}\circ t$, we
have that $R_{(g)_{Z}}^{\ast}\tilde{\mathbb{A}}=\mathrm{Ad}_{(g)_{Z}^{-1}}%
\tilde{\mathbb{A}}$, and $\tilde{\mathbb{A}}$ is a connection form.

Finally, since the actions of $G$ and $\mathcal{G}$ commute on $P\times%
\mathcal{A}$, the invariance of $\mathbb{A}$ with respect to $\mathcal{G}$
induces the invariance of $\tilde{\mathbb{A}}$ with respect to $\tilde{%
\mathcal{G}}$.
\end{proof}

Now we can apply all the preceding results, but we should take $p\in
\mathcal{I}_{\mathbb{Z}}^{r}(\tilde{G})$ and $\mu\in H^{2r}(B\widetilde {G},%
\mathbb{Z})$ in place of the corresponding objects for $G$.

As $\widetilde{\mathbb{A}}$ is $\widetilde{\mathcal{G}}$-invariant, if $%
\mathfrak{A}$ is a connection on $\widetilde{\mathcal{A}}\rightarrow
\widetilde{\mathcal{A}}/\widetilde{\mathcal{G}}$, by Proposition \ref%
{quotcha} they determine a connection on the quotient bundle $\widetilde{%
\mathbb{P}}/\widetilde{\mathcal{G}}\rightarrow M\times \widetilde{\mathcal{A}%
}/\widetilde{\mathcal{G}}$, and we have the Chern-Simons characters $%
\widetilde{\chi }_{\mathfrak{A}}\colon Z_{2r-1}(M\times \widetilde{\mathcal{A%
}}/\widetilde{\mathcal{G}})\rightarrow \mathbb{R}/\mathbb{Z}.$ Doing the
same as in Section \ref{connections}\ we obtain maps $\widetilde{\chi }_{%
\mathfrak{A}}^{k}\colon Z_{2r-k-1}(M)\times Z_{k}(\widetilde{\mathcal{A}}/%
\widetilde{\mathcal{G}})\rightarrow \mathbb{R}/\mathbb{Z}$.

If we restrict to $\widetilde{\mathcal{F}}=\{A\in\mathcal{A}:F_{A}=0$, $A$
is irreducible$\}\subset\widetilde{\mathcal{A}}$ the space of irreducible
flat connections, using the the same arguments as in Theorem \ref{th}, we
get cohomology maps $\widetilde{\chi}^{k}\colon H_{2r-k-1}(M)\times H_{k}(%
\widetilde{\mathcal{F}}/\widetilde{\mathcal{G}})\rightarrow\mathbb{R}/%
\mathbb{Z}$ that does not depend on the connection $\mathfrak{A}$

\bigskip

We recall that if $G=U(n)$ or $SU(n)$, then $\widetilde{G}=PU(n)=SU(n)/%
\mathbb{Z}_{n}$ is the projective unitary group. We have $\mathfrak{su}(n)=%
\mathfrak{pu}(n)$ and $\mathcal{I}^{r}(PU(n))=\mathcal{I}^{r}(SU(n))$ as
they are connected groups. However, we have $\mathcal{I}_{\mathbb{Z}%
}^{r}(PSU(n))\varsubsetneq \mathcal{I}_{\mathbb{Z}}^{r}(SU(n))$. The
projection $i\colon SU(n)\rightarrow PU(n)$ induces a map $Bi^{\ast }\colon
H^{\bullet }(\mathbf{B}PU(n),\mathbb{Z})\rightarrow $ $H^{\bullet }(\mathbf{B%
}SU(n),\mathbb{Z})$. It is shown in \cite{AW} that $H^{4}(\mathbf{B}PU(n),%
\mathbb{Z})\simeq \mathbb{Z}$ has a generator $e_{2}$ such that $Bi^{\ast
}(e_{2})=2nc_{2}$ for $n$ even and $B\rho ^{\ast }e_{2}=nc_{2}$ for $n$ odd,
where $c_{2}$ denotes the second Chern class. Hence if $C_{2}\in \mathcal{I}%
_{\mathbb{Z}}^{2}(SU(n))$ is the second Chern polynomial we have $%
C_{2}\notin \mathcal{I}_{\mathbb{Z}}^{2}(PU(n))$ but $2nC_{2}\in \mathcal{I}%
_{\mathbb{Z}}^{2}(PU(n))$. We conclude that for a $SU(n)$-principal bundle $%
(C_{2},c_{2})$ does not determine cohomology classes on $\widetilde{\mathcal{%
F}}/\widetilde{\mathcal{G}}$, but $(2nC_{2},e_{2})$ does.

\section{Low dimensional cases}

If $c\in Z_{2r-k-1}$, we define $\xi _{c,\underline{\mathfrak{A}}}\in \hat{H}%
^{k+1}(\mathcal{A}/\mathcal{G})$ by $\xi _{c,\underline{\mathfrak{A}}%
}(s)=\chi _{\underline{\mathfrak{A}}}^{k}(c\times s)$. In this Section we
study the geometrical interpretation of the differential characters $\xi _{c,%
\underline{\mathfrak{A}}}$ for $k=0$ and $k=1$.

\subsection{First order characters and Dijkgraaf-Witten action for
Chern-Simons theory\label{1orderIntCS}}

We recall the definition of the Dijkgraaf-Witten action for Chern-Simons
theory. Usually classical Chern-Simons theory is defined for trivial bundles
over 3-manifolds using global sections, but this procedure can not be
generalized to non-trivial bundles. In \cite{DW} Dijkgraaf and Witten shown
how the Chern-Simons action can be defined in this case using differential
characters.

Let $A$ be a connection on $P\rightarrow M$, $p\in \mathcal{I}_{\mathbb{Z}%
}^{r}(G)$, $\Upsilon $ a compatible characteristic class and let $\chi
_{A}\in \hat{H}^{2r}(M)$ be the Chern-Simons character of $A$ associated to $%
p$ and $\Upsilon $. If $c\colon C\rightarrow M$ is a smooth map with $C$
closed and $\dim C=2r-1$ we define the Chern-Simons action $\lambda
_{c}\colon \mathcal{A}\rightarrow \mathbb{R}/\mathbb{Z}$ by $\lambda
_{c}(A)=\chi _{A}(c)$. By\ definition, if $c=\partial u$ we have $\lambda
_{\partial u}(A)=\chi _{A}(\partial u)=\int_{u}p(F,\ldots ,F)$. If $A_{0}$
is another connection on $P$ we have $\lambda _{c}(A)-\lambda
_{c}(A_{0})=\chi _{A}(c)-\chi _{A_{0}}(c)=\int_{c}Tp(A,A_{0})$ and hence $%
\lambda _{c}(A)=\lambda _{c}(A_{0})+\int_{c}Tp(A,A_{0})$. Note that $\lambda
_{c}$ gives a definition of the Chern-Simons action which is independent of
the background connection $A_{0}$ chosen, i.e., the characteristic class $%
\Upsilon $ fixes the value of $\lambda _{c}(A_{0})$. As remarked in \cite{DW}
this constant is relevant in the quantization of the theory. It can be seen
that $\lambda _{c}$ is $\mathcal{G}$-invariant.

Now we show how first order characters are related to the Chern-Simons
Lagrangian $\lambda _{c}$.

A differential character of order 1 is a homomorphism $\chi \colon
Z_{0}(M)\rightarrow \mathbb{R}/\mathbb{Z}$ such that there exist $\lambda
\in \Omega ^{1}(M)$ with $\chi (\{x\}-\{y\})=\int_{\gamma }\lambda $ for any
$\gamma \subset M$, $\partial \gamma =x-y$. We define $\varphi _{\chi
}\colon M\rightarrow \mathbb{R}/\mathbb{Z}$ by $\varphi _{\chi }(x)=\chi
(\{x\})$. It can be seen that $\varphi $ is differentiable (see \cite{BB}).
Conversely, given any map $\varphi \colon M\rightarrow \mathbb{R}/\mathbb{Z}$%
, we can define $\chi _{\varphi }\colon Z_{0}(M)\rightarrow \mathbb{R}/%
\mathbb{Z}$ by setting $\chi _{\varphi }(\{x\})=\varphi (x)$, and $\lambda
=\varphi ^{\ast }(dt)$ where $t$ is a coordinate on $\mathbb{R}$. Hence, a
differential character of order 1 is simply a map $\varphi \colon
M\rightarrow \mathbb{R}/\mathbb{Z}$.

\begin{lemma}
The character $\xi _{c,\underline{\mathfrak{A}}}\in \hat{H}^{1}(\mathcal{A}/%
\mathcal{G})$ does not depend on $\mathfrak{A}$.
\end{lemma}

\begin{proof}
By equation (\ref{CambioC}) we have $\xi _{c,\underline{\mathfrak{A}}%
}([A])=\xi _{c,\underline{\mathfrak{A}}}(s)+\int_{c\times \lbrack A]}Tp(%
\underline{\mathfrak{A}}^{\prime },\underline{\mathfrak{A}})^{2r-1,0}$. But $%
Tp(\underline{\mathfrak{A}}^{\prime },\underline{\mathfrak{A}})^{2r-1,0}=0$
because $\underline{\mathfrak{A}}^{\prime }-\underline{\mathfrak{A}}\in
\Omega ^{0,1}(M\times \mathcal{A}/\mathcal{G},\mathrm{ad}(P\times \mathcal{A}%
/\mathcal{G)})$.
\end{proof}

As $\xi _{c,\underline{\mathfrak{A}}}\in \hat{H}^{1}(\mathcal{A}/\mathcal{G})
$ does not depend on $\mathfrak{A}$, we denote $\xi _{c,\underline{\mathfrak{%
A}}}$ simply by $\xi _{c}$. The first order character $\xi _{c}$ determines
a function $\varphi _{c}\colon \mathcal{A}/\mathcal{G}\rightarrow \mathbb{R}/%
\mathbb{Z}$ given by $\varphi _{c}([A])=\xi _{c}([A])=\chi _{\underline{%
\mathfrak{A}}}(c\times \lbrack A])$. The connection $A$ induces a morphism
of principal $G$-bundles $g_{A}\colon P\rightarrow (P\times \mathcal{A})/%
\mathcal{G}$, defined by $g_{A}(y)=[(y,A)]$ for $y\in P$. By Proposition \ref%
{quotcha}\ we have $g_{A}^{\ast }(\mathfrak{A})=A$, and hence $\underline{g}%
_{A}^{\ast }\chi _{\underline{\mathfrak{A}}}=\chi _{A}$, where $\underline{g}%
_{A}\colon M\rightarrow M\times (\mathcal{A}/\mathcal{G})$ is given by $%
\underline{g}_{A}(x)=(x,[A])$. We conclude that $\lambda _{c}(A)=\chi
_{A}(c)=\underline{g}_{A}^{\ast }\chi _{\underline{\mathfrak{A}}}(c)=\chi _{%
\underline{\mathfrak{A}}}(c\times \lbrack A])=\varphi _{c}([A])$. In
particular this proves that $\lambda _{c}$ is $\mathcal{G}$-invariant.

We study the restriction to the moduli space of flat connections. Equations (%
\ref{mainobject}) and (\ref{mainobject2}) imply that if $A$ is flat then for
$r\geq 1$ the Chern-Simons action $\lambda _{c}(A)$ only depends on the
homology class of $c$. Also, for $r\geq 2$, $\lambda _{c}(A)$ only depends
on the connected component of $[A]$ on $\mathcal{F}/\mathcal{G}$ (i.e, $%
\lambda _{c}(A)$ determines a locally constant function on $\mathcal{F}/%
\mathcal{G}$). For example, when $\dim M=3$ and $c=M$ this fact is well
known, as the extremals of the Chern-Simons action are the flat connections,
and hence the Chern-Simons action is constant on each connected component of
$\mathcal{F}$.

\subsection{Differential characters of order 2\label{2orden}}

First we recall that if $\mathcal{U}\rightarrow M$ a principal $U(1)$-bundle
with connection $\Theta $, and curvature $\omega \in \Omega ^{2}(M)$, the
log-holonomy of $\Theta $ $\log \mathrm{hol}_{\Theta }\colon
Z_{1}(M)\rightarrow \mathbb{R}/\mathbb{Z}$ is a differential character of
order 2 with curvature $\omega $. A classical result in differential
cohomology asserts that the converse is true, that is,\ every second order\
differential character can be represented\ as the holonomy of a connection $%
\Theta $ on a principal $U(1)$ bundle $\mathcal{U}\rightarrow M$ and this
bundle and connection are unique modulo isomorphisms. In our case is
possible to give a concrete bundle and connection using the following result
proved in \cite{CSconnections}

\begin{theorem}
\label{Prop}Let $N$ be a connected and simply connected manifold in which $%
\mathcal{G}$ acts in such a way that $\pi \colon N\rightarrow N/\mathcal{G}$
is a principal $\mathcal{G}$-bundle. Let $\chi \in \hat{H}^{2}(N/\mathcal{G}%
) $ be a second order differential character on $N/\mathcal{G}$ with
curvature $\omega $, and assume that there exist\ $\lambda \in \Omega ^{1}(N)
$ such that $\pi ^{\ast }\omega =d\lambda $. Then there exists a unique lift
of the action of $\mathcal{G}$ to $N\times U(1)$ by $U(1)$-bundle
automorphism such that $\Theta =\vartheta -2\pi i\lambda \in \Omega
^{1}(N\times U(1),i\mathbb{R})$ is projectable onto a connection $\underline{%
\Theta }$ on $\mathcal{U}=(N\times U(1))/\mathcal{G}\rightarrow N/\mathcal{G}
$ and $\chi =\log \mathrm{hol}_{\underline{\Theta }}$ ($\vartheta $ denotes
the Maurer-Cartan form on $U(1)$).

The action of $\phi \in \mathcal{G}$ on $N\times U(1)\mathbb{\ }$is given by
$\Phi _{\phi }\colon N\times U(1)\rightarrow N\times U(1),$ $\Phi _{\phi
}(x,u)=(\phi x,\exp (2\mathrm{\pi }i\alpha _{\phi }(x))\cdot u)$, where $%
\alpha _{\phi }\colon N\rightarrow \mathbb{R}/\mathbb{Z}$ is defined by $%
\alpha _{\phi }(x)=\int\nolimits_{\gamma }\lambda -\chi (\pi \circ \gamma )$%
, and $\gamma $ is any curve on $N$ joining $x$ and $\phi x$.
\end{theorem}

If $c\in Z_{2r-2}(M)$ we have $\xi _{c,\underline{\mathfrak{A}}}\in \hat{H}%
^{2}(\mathcal{A}/\mathcal{G})$. By Section \ref{2equi} $\varpi
_{c}=\int_{c}p_{\mathcal{G},\mathbb{A}}\in \Omega _{\mathcal{G}}^{2}(%
\mathcal{A}/\mathcal{G})$\ can be written $\varpi _{c}=\sigma _{c}+\mu _{c}$%
, and by the definition of $p_{\mathcal{G},\mathbb{A}}$ we have $\sigma
_{c}=\int_{c}p(\mathbb{F})$ and $\mu _{c}(X)=-r\int_{c}p(\mathrm{pr}%
_{1}^{\ast }X,\mathbb{F},\overset{(r-1}{\ldots },\mathbb{F})$. If $A_{0}$ is
a connection on $P\rightarrow M$, we define $\rho _{c}=\int_{c}Tp(\mathbb{A},%
\overline{\mathrm{pr}}_{1}^{\ast }A_{0})\in \Omega ^{1}(\mathcal{A})$, where
$\overline{\mathrm{pr}}_{1}\colon P\times \mathcal{A}\rightarrow P\,$denotes
the\ projection. If we set $\lambda _{c}=\rho _{c}+\mu _{c}(\mathfrak{A})$,
we have the following

\begin{proposition}
We have $\mathrm{curv}(\xi _{c,\underline{\mathfrak{A}}})=\int_{c}p(%
\underline{\mathfrak{F}})$ and $\pi ^{\ast }(\int_{c}p(\underline{\mathfrak{F%
}}))=d\lambda _{c}$.
\end{proposition}

\begin{proof}
For any $s\in Z_{1}(\mathcal{A}/\mathcal{G})$ we have $\xi _{c,\underline{%
\mathfrak{A}}}(\partial s)=\chi _{\underline{\mathfrak{A}}}(c\times \partial
s)=\chi _{\underline{\mathfrak{A}}}(\partial (c\times s))=\int_{c\times s}p(%
\underline{\mathfrak{F}})=\int_{s}\int_{c}p(\underline{\mathfrak{F}})$, and
by the definition of the curvature of a differential character we conclude
that $\mathrm{curv}(\xi _{c,\underline{\mathfrak{A}}})=\int_{c}p(\underline{%
\mathfrak{F}})$.

Furthermore, we have $d\rho _{c}=\int_{c}dTp(\mathbb{A},\overline{\mathrm{pr}%
}_{1}^{\ast }A_{0})=\int_{c}p(\mathbb{F})-\int_{c}\mathrm{pr}_{1}^{\ast
}p(F_{0})=\sigma _{c}$\ as $\int_{c}\mathrm{pr}_{1}^{\ast }p(F_{0})=0$
because $\mathrm{pr}_{1}^{\ast }p(F_{0})\in \Omega ^{2r-1,0}$. By equation %
\ref{F2} we have $\int_{c}p(\underline{\mathfrak{F}})=\mathrm{C}_{\mathfrak{A%
}}(\int_{c}p_{\mathcal{G}})\in \Omega _{\mathcal{G}}^{2}(\mathcal{A}/%
\mathcal{G})$, and the result follows from Lemma \ref{2equiLemma}.
\end{proof}

By applying Proposition \ref{Prop} to the bundle $\pi _{\mathcal{A}}\colon
\mathcal{A}\rightarrow \mathcal{A}/\mathcal{G}$ and the character $\xi _{c,%
\underline{\mathfrak{A}}}$ we obtain the following

\begin{proposition}
\label{PropIntCS} Let $A_{0}$ be a background connection on $P\rightarrow M$
and $\mathfrak{A}$ be a connection on $\mathcal{A}\rightarrow \mathcal{A}/%
\mathcal{G}$. Then there exists a unique lift of the action of $\mathcal{G}$
on $\mathcal{A}$ to an action on $\mathcal{A}\times U(1)$ by $U(1)$-bundle
automorphisms such that $\Theta _{c}=\vartheta -2\pi i(\rho _{c}+\mu _{c}(%
\mathfrak{A}))\in \Omega ^{1}(\mathcal{A}\times U(1),i\mathbb{R})$ is
projectable onto a connection $\underline{\Theta }_{c}$ on $\mathcal{U}_{c}=(%
\mathcal{A}\times U(1))/\mathcal{G}\rightarrow \mathcal{A}/\mathcal{G}$ and $%
\xi _{c,\underline{\mathfrak{A}}}=\log \mathrm{hol}_{\underline{\Theta }_{c}}
$.

The action of $\phi \in \mathcal{G}$ on $\mathcal{A}\times U(1)\mathbb{\ }$%
is given by $\Phi _{\phi }\colon \mathcal{A}\times U(1)\rightarrow \mathcal{A%
}\times U(1),$ $\Phi _{\phi }(A,u)=(\phi A,\exp (2\mathrm{\pi }i\alpha
_{\phi }(A))\cdot u)$, where $\alpha _{\phi }\colon \mathcal{A}\rightarrow
\mathbb{R}/\mathbb{Z}$ is defined by $\alpha _{\phi
}(A)=\int\nolimits_{\gamma }\lambda -\chi (\pi _{\mathcal{A}}\circ \gamma )$%
, and $\gamma $ is any curve on $\mathcal{A}$ joining $A$ and $\phi A$.
\end{proposition}

\begin{remark}
As the characteristic class of the holonomy of a connection coincides with
the Chern class of the bundle we also have $c_{1}(\mathcal{U}_{c})=\mathrm{%
char}(\xi _{c,\underline{\mathfrak{A}}})=\Upsilon _{(P\times \mathcal{A})/%
\mathcal{G}}/c$ where $\Upsilon _{(P\times \mathcal{A})/\mathcal{G}}\in
H^{2r}(M\times \mathcal{A}/\mathcal{G},\mathbb{Z})$ is the characteristic
class of $(P\times \mathcal{A})/\mathcal{G}\rightarrow M\times \mathcal{A}/%
\mathcal{G}$ associated to $\Upsilon $, and $/$ denotes cap product.
\end{remark}

In place of the principal $U(1)$-bundle $\mathcal{U}_{c}$, we can consider
the associated Hermitian line bundle $\mathcal{L}_{c}$. The connection $%
\underline{\Theta }_{c}$ determines a hermitian connection $\nabla ^{%
\underline{\Theta }_{c}}$ on this bundle.

The bundle $\mathcal{U}_{c}$ is defined using a connection $\mathfrak{A}$ in
$\mathcal{A}\rightarrow \mathcal{A}/\mathcal{G}$, but we have the following

\begin{proposition}
\label{changeChernSimons}The map $\alpha $ is independent of the connection $%
\mathfrak{A}$ chosen in $\mathcal{A}\rightarrow \mathcal{A}/\mathcal{G}$.
\end{proposition}

\begin{proof}
If $\mathfrak{A}^{\prime }$ is another connection on $\mathcal{A}\rightarrow
\mathcal{A}/\mathcal{G}$ we have $\lambda _{c}^{\prime }=\rho _{c}+\mu _{c}(%
\mathfrak{A}^{\prime })=\lambda _{c}+\mu _{c}(\mathfrak{A}^{\prime }-%
\mathfrak{A})$ and by equation (\ref{CambioC}) for any curve $\gamma $ on $N$
such that $\pi _{\mathcal{A}}(\gamma (0))=\pi _{\mathcal{A}}(\gamma (1))$ we
have $\xi _{c,\underline{\mathfrak{A}}}(\pi _{\mathcal{A}}\circ \gamma )=\xi
_{c,\underline{\mathfrak{A}}}(\pi _{\mathcal{A}}\circ \gamma )+\int_{c\times
\pi _{\mathcal{A}}\circ \gamma }Tp(\underline{\mathfrak{A}}^{\prime },%
\underline{\mathfrak{A}})=\xi _{c,\underline{\mathfrak{A}}}(\pi \circ \gamma
)+\int_{c\times \gamma }Tp(\pi _{\mathcal{G}}^{\ast }\underline{\mathfrak{A}}%
^{\prime },\pi _{\mathcal{G}}^{\ast }\underline{\mathfrak{A}})$, where $\pi
_{\mathcal{G}}\colon M\times \mathcal{A}\rightarrow M\times \mathcal{A}/%
\mathcal{G}$ is the projection. But $Tp(\pi ^{\ast }\underline{\mathfrak{A}}%
^{\prime },\pi ^{\ast }\underline{\mathfrak{A}})=r\int_{0}^{1}p(\pi ^{\ast }(%
\underline{\mathfrak{A}}^{\prime }-\mathfrak{\underline{\mathfrak{A}}}%
),F_{t},\ldots ,F_{t})=r\int_{0}^{1}p(\mathbb{A}(\mathfrak{A}^{\prime }-%
\mathfrak{A}),F_{t},\ldots ,F_{t})$, where $F_{t}$ is the curvature of $%
A_{t}=\mathbb{A}-\mathbb{A}((1-t)\mathfrak{A}^{\prime }+t\mathfrak{A})$. By
Lemma\ \ref{Lema1}\ we have $F_{t}^{2,0}=\mathbb{F}^{2,0}$ and hence $%
\int_{c\times \gamma }Tp(\pi ^{\ast }\underline{\mathfrak{A}}^{\prime },\pi
^{\ast }\underline{\mathfrak{A}})=\int_{c\times \gamma }Tp(\pi ^{\ast }%
\underline{\mathfrak{A}}^{\prime },\pi ^{\ast }\underline{\mathfrak{A}}%
)^{1,2r-2}=\int_{c\times \gamma }r\int_{0}^{1}p(\mathbb{A}(\mathfrak{A}%
^{\prime }-\mathfrak{A}),\mathbb{F},\ldots ,\mathbb{F})=\int_{\gamma }\mu
_{c}(\mathfrak{A}^{\prime }-\mathfrak{A})$. The result follows as we have $%
\alpha _{\phi }^{\prime }(A)=\int\nolimits_{\gamma }\lambda ^{\prime }-\xi
_{c,\underline{\mathfrak{A}}}(\pi \circ \gamma )=\int\nolimits_{\gamma
}\lambda +\int_{\gamma }\mu _{c}(\mathfrak{A}^{\prime }-\mathfrak{A})-\xi
_{c,\underline{\mathfrak{A}}}(\pi \circ \gamma )-\int_{\gamma }\mu _{c}(%
\mathfrak{A}^{\prime }-\mathfrak{A})=\alpha _{\phi }(A)$.
\end{proof}

\bigskip

These bundles are extended to the case of non-free actions and also to the
action of automorphisms in \cite{CSconnections}. It is proved that bundles
for different $A_{0}$ are canonically isomorphic and that these bundles
generalize the Chern Simons lines as if $c=\partial u$, then the
Chern-Simons action $S_{u}(A)=\exp (-2\mathrm{\pi }i\cdot
\int_{u}Tp(A,A_{0}))$ determines a section of $\mathcal{U}_{c}$.

Our interest is in the restriction of $\mathcal{U}_{c}$ to the moduli space
of flat connections. For $r\geq 2$ we have $\mathcal{F}\subset \mu
_{c}^{-1}(0)$. In particular the restriction $\theta _{c}=\iota _{\mathcal{F}%
}^{\ast }\underline{\Theta }_{c}$ to $\mathcal{F}\times U(1)$ of the form $%
\underline{\Theta }_{c}$ does not depend on $\mathfrak{A}$ as $\mu _{c}(%
\mathfrak{A})=0$ on $\mathcal{F}$. We conclude that the restriction of the
bundle $\mathcal{U}_{c}\rightarrow \mathcal{A}/\mathcal{G}$ and of the
connection $\underline{\Theta }_{c}$ do not depend on $\mathfrak{A}$, in
accordance with Theorem \ref{th}.

The curvature of $\theta _{c}$ is the pre-symplectic form $\underline{\sigma
}_{c}$ on $\mathcal{F}/\mathcal{G}$\ obtained by symplectic reduction to $%
\mathcal{F}$ from $(\mathcal{A},\varpi _{c})$. For example, if $M$\ is a
closed oriented surface we can take $c=M$, $G=SU(2)$, $P$ the trivial $SU(2)$%
-bundle, $p(X)=\frac{1}{8\pi ^{2}}\mathrm{tr}(X^{2})$ and $\Upsilon $ the
second Chern class then $\sigma _{M}$ and $\mu _{M}$ coincide with the
symplectic structure and moment map defined on \cite{AB1} (e.g. see \cite%
{equiconn}). Furthermore, it is shown in \cite{CSconnections} that the
bundle $\mathcal{L}_{M}\rightarrow \mathcal{F}/\mathcal{G}$\ is isomorphic
to Quillen determinant line bundle. Hence, our character $\xi _{M}$ is the
holonomy of a connection $\theta _{M}$\ on the determinant line bundle $\det
\overline{\partial }\rightarrow \mathcal{F}/\mathcal{G}$.

If $\dim M>2$, for each submanifold $c$\ of dimension $2$ we have a line
bundle $\mathcal{L}_{c}\rightarrow \mathcal{F}/\mathcal{G}$ with connection $%
\theta _{c}$. A similar family of bundles $\ell _{c}\rightarrow \mathcal{A}/%
\mathcal{G}$ is constructed for example in \cite[Section 5.2.1]{Donaldson}
associated to surfaces $c$ immersed on a $4$-manifold $M$. These bundles are
defined as determinant line bundles of Dirac operators, and are isomorphic
to our bundles\ as they have the same integer Chern class $c_{1}(\ell
_{c})=\Upsilon _{(P\times \mathcal{A})/\mathcal{G}}/c=c_{1}(\mathcal{L}%
_{c})\in H^{2}(\mathcal{A}/\mathcal{G},\mathbb{Z})$.

From equation (\ref{mainobject2}) we conclude that the holonomy of the
connection $\theta _{c}$ only depends on the homology class of $c$.

If $r\geq 3$ we have $\sigma _{c}|_{\mathcal{F}}=0$, and in this case $%
\theta _{c}$ is a flat connection on $\mathcal{L}_{c}\rightarrow \mathcal{F}/%
\mathcal{G}$, and hence its holonomy defines a cohomology class in $H^{1}(%
\mathcal{F}/\mathcal{G},\mathbb{R}/\mathbb{Z})$, in accordance with our
result of equation (\ref{mainobject}). Furthermore this equation implies
that the cohomology class only depends on the homology class of $c$ on $M$.

\section{Particular cases\label{Sectparticular}}

The main objects $\chi ^k$ of this work are close to some constructions
found in the literature:

\paragraph{The case $k=\mathrm{dim}M-2r+1$.}

If we choose a cohomology class $c\in H_{2r-k-1}(M)$, for $k\leq r-1$ we
obtain cohomology classes $\xi _{c}\in H^{k}(\mathcal{F}/\mathcal{G},\mathbb{%
R}/\mathbb{Z})$, $\xi _{c}(s)=\chi ^{k}(c,s)$. When $M$ is a compact
manifold without boundary of dimension $\dim M=2r-k-1$, for $k\leq r-1$, the
cohomology classes $\xi _{M}$ coincide with some classes defined on \cite%
{DupontKamber}. In that article, the authors work with families of
connections instead of $\mathcal{A}/\mathcal{G}$, and Deligne cohomology
instead of differential characters. Note that our result is more general in
the sense that we obtain cohomology classes also for submanifolds $c\subset M
$.

\paragraph{The unitary group.}

For $G=U(n)$, we consider the map (\ref{mainobject}) for the choice of $p_{r}
$ and $\mu _{r}$ as $r$-th Chern polynomial and class, $k\leq r-2$,
respectively. Furthermore, we regard that map in the cohomology level as
\begin{equation*}
\left. \chi _{p_{r}}^{k}\right\vert _{\mathcal{F}/\mathcal{G}}\colon H_{k}(%
\mathcal{F}/\mathcal{G})\rightarrow H^{2r-k-1}(M,\mathbb{R}/\mathbb{Z)}.
\end{equation*}%
By adding the maps for different $r$, we obtain the map $\tilde{\chi}^{k}={%
\textstyle\bigoplus\limits_{r=k+2}^{n}}\left. \chi _{p_{r}}^{k}\right\vert _{%
\mathcal{F}/\mathcal{G}}$,
\begin{equation*}
\tilde{\chi}^{k}\colon H_{k}(\mathcal{F}/\mathcal{G})\rightarrow {\textstyle%
\bigoplus\limits_{r=k+2}^{n}}H^{2r-k-1}(M,\mathbb{R}/\mathbb{Z)},
\end{equation*}%
which is a similar to the construction in \cite{iyer}, by a completely
different method, for the study of cohomological invariants of variations of
flat connections.

\noindent (M.C.L) \textsc{ICMAT (CSIC-UAM-UC3M-UCM)\newline
Departamento de \'Algebra, Geometr\'{\i}a y Topolog\'{\i}a, Facultad de Matem%
\'{a}ticas, Universidad Complutense de Madrid, 28040-Madrid, Spain}

\noindent \emph{E-mail:\/} \texttt{mcastri@mat.ucm.es}

\medskip

\noindent (R.F.P) \textsc{Departamento de Econom\'\i a Financiera y
Contabilidad I. Facultad de Ciencias Econ\'omicas y Empresariales,
Universidad Complutense de Madrid, 28223-Madrid, SPAIN}

\noindent \textit{E-mail:\/} \texttt{roferreiro@ccee.ucm.es}

\end{document}